\documentclass[12pt]{article}

\usepackage[latin1]{inputenc}
\usepackage{fancyhdr}
\usepackage{indentfirst}
\usepackage{graphicx}
\usepackage{newlfont}

\usepackage{subfigure}

\usepackage{epsfig}
\usepackage{amssymb}
\usepackage{amsmath}
\usepackage{latexsym}
\usepackage{amsthm}
\usepackage{bbold}

\usepackage{xcolor}

\usepackage[export]{adjustbox}

\newtheorem{theorem}{Theorem}[section]
 
\newtheorem{proposition}{Proposition}[section]
\newtheorem{lemma}{Lemma}[section]

\newtheorem{definition}{Definition}[section]

\newtheorem{remark}{Remark}[section]

\numberwithin{equation}{section}

\begin{document}

\title{Homogenization of Smoluchowski-type equations with transmission 
boundary conditions}

\author{Bruno Franchi$^1$ and Silvia Lorenzani$^2$ }

\date{\em $^{1}$ Dipartimento di Matematica, Universit\`a di Bologna,
Piazza di Porta S. Donato 5, 40126 Bologna, Italy \\
$^{2}$ Dipartimento di Matematica, Politecnico di Milano, Piazza Leonardo da
Vinci 32, 20133 Milano, Italy}

\maketitle

\begin{abstract}
In this work, we prove a two-scale homogenization result for a set
of diffusion-coagulation Smoluchowski-type equations with transmission
boundary conditions. This system is meant to describe the aggregation and
diffusion of pathological tau proteins in the cerebral tissue, a process
associated with the onset and evolution of a large variety of tauopathies
(such as Alzheimer's disease). We prove the existence, positivity and
boundedness of solutions to the model equations derived at the microscale
(that is the scale of single neurons).
Then, we study the convergence of the homogenization process to the solution
of a macro-model asymptotically consistent with the microscopic one.

\end{abstract}



\section{Introduction}\label{sec1}

Let us consider a bounded open set
$\Omega$ in $\mathbb{R}^3$ with a smooth boundary $\partial \Omega$.
This domain is decomposed into long cylindrical cavities,
periodically distributed with period $\epsilon$, having generators parallel
to the $z$-axis.
More precisely, given a bounded domain $D$ in $\mathbb{R}^2$,
let us set $\Omega:=D \times ]0,L[$ with $x \in D$ and $z \in ]0,L[$.
We denote by $\Gamma_L:=\partial D \times ]0,L[$ the lateral boundary
and by $\Gamma_B:=\overline{D} \times \{0,L\}$ the upper
and lower sides of $\partial \Omega$.
Let $Y=[0,1] \times [0,1]$ be the unit periodicity cell in $\mathbb{R}^2$
having the paving property, i.e. the disjoint union of translated copies of
$Y$ can indeed cover the whole space.
Let $X$ be an open subset of $Y$
with a smooth boundary $R=\partial X$, such that $\overline {X} \subset 
\mathrm{Int}\, Y$, and $Z=Y \setminus X$.
Then, we define $G_{\epsilon}$ to be the set of all translated images of
$\epsilon \, \overline{X}$:

$$G_{\epsilon}:=\cup \bigg\{ \epsilon (k+\overline {X}) \mid \epsilon (k+
\overline {X}) \subset D, k \in \mathbb {Z}^2 \bigg\}$$
and $D_{\epsilon}=D \setminus G_{\epsilon}$, as well as

$$R_{\epsilon}:=\cup \bigg\{ \partial (\epsilon (k+\overline {X})) \mid
\epsilon (k+\overline {X}) \subset D, k \in \mathbb {Z}^2 \bigg\}.$$
Thus, the geometric structure of the domain in $\mathbb{R}^3$ is given by:
$\Pi_{\epsilon}:=G_{\epsilon} \times ]0,L[,$
$\Omega_{\epsilon}:=D_{\epsilon} \times ]0,L[,$ and
$\Gamma_{\epsilon}:=R_{\epsilon} \times ]0,L[.$

Throughout this paper, $\epsilon$ will denote the general term of a
sequence of positive reals which converges to zero.
Let us introduce two nonnegative vector-valued functions,
$u^{\epsilon}: [0,T] \times \Pi_{\epsilon} \rightarrow \mathbb{R}^M$,
$u^{\epsilon} (t,x,z)=(u_1^{\epsilon}, \ldots, u_M^{\epsilon})$ and
$v^{\epsilon}: [0,T] \times \Omega_{\epsilon} \rightarrow \mathbb{R}^M$,
$v^{\epsilon} (t,x,z)=(v_1^{\epsilon}, \ldots, v_M^{\epsilon})$, which
solve the following system of discrete Smoluchowski-type equations
\cite{Smo} with
transmission boundary conditions.

For $1 \leq m \leq M$ we have:

\begin{eqnarray} \label{2.1a}
\begin{cases}
\frac{\displaystyle \partial{u_m^{\epsilon}}}{\displaystyle \partial t}-
{\epsilon}^2 \, D_m \,
\Delta_x u_m^{\epsilon}-\tilde{D}_m \, \partial_z^2 u_m^{\epsilon} \\
=L_m(u^{\epsilon})+f_m^{\epsilon}(t,x,z)
 & \text{in } [0,T] \times \Pi_{\epsilon} \\

\\ 
\frac{\displaystyle \partial{v_m^{\epsilon}}}{\displaystyle \partial t}-
d_m \, \Delta v_m^{\epsilon}= N_m(v^{\epsilon})
 & \text{in } [0,T] \times \Omega_{\epsilon} \\

\\

\\
\epsilon \, D_m \, \nabla_x u_m^{\epsilon} 
\cdot \nu_{\epsilon}=
-c_m(x,z) (u_m^{\epsilon}-v_m^{\epsilon})_{+} & \text{on } 
[0,T] \times \Gamma_{\epsilon}  \\

\\
d_m \, \nabla_x v_m^{\epsilon} 
\cdot \nu_{\epsilon}={\epsilon}^2 \, D_m \, \nabla_x u_m^{\epsilon}
\cdot \nu_{\epsilon}
 & \text{on } 
[0,T] \times \Gamma_{\epsilon}  \\

\\
\nabla_x v_m^{\epsilon}
\cdot \nu_{\epsilon}=0 & \text{on } [0,T] \times 
\Gamma_L  \\

\\
\partial_z u_m^{\epsilon}=\partial_z v_m^{\epsilon}=0 &
\text{on } [0,T] \times
\Gamma_B  \\

\\
u_m^{\epsilon}(0,x,z)= U_m^{\epsilon}(x,z) & \text{in } \Pi_{\epsilon}
\\
v_m^{\epsilon}(0,x,z)= 0 & \text{in } \Omega_{\epsilon}
\end{cases}
\end{eqnarray}
where

\begin{equation} \label{2.2a}
L_1(u^{\epsilon})=-\sum_{j=1}^M a_{1,j} u_1^{\epsilon} u_j^{\epsilon}
\end{equation}

\begin{equation} \label{2.2b}
L_m(u^{\epsilon})=\frac{1}{2} \, \sum_{j=1}^{m-1} a_{j,m-j} \, 
u_j^{\epsilon} \, u_{m-j}^{\epsilon}
-\sum_{j=1}^M a_{m,j} u_m^{\epsilon} u_j^{\epsilon}
\; \; \; (1 < m < M)
\end{equation}

\begin{equation} \label{2.2c}
L_M(u^{\epsilon})= \frac{1}{2} \, \sum_{\substack{j+k \geq M \\ 
k< M \\ j<M}} 
a_{j,k} \, u_j^{\epsilon} \, u_k^{\epsilon} 
\end{equation}

\begin{equation} \label{2.3a}
N_1(v^{\epsilon})=-\sum_{j=1}^M b_{1,j} v_1^{\epsilon} v_j^{\epsilon}
\end{equation}

\begin{equation} \label{2.3b}
N_m(v^{\epsilon})=\frac{1}{2} \, \sum_{j=1}^{m-1} b_{j,m-j} \, v_j^{\epsilon} 
\, v_{m-j}^{\epsilon}
-\sum_{j=1}^M b_{m,j} v_m^{\epsilon} v_j^{\epsilon}
\; \; \; (1 < m < M)
\end{equation}

\begin{equation} \label{2.3c}
N_M(v^{\epsilon})=\frac{1}{2} \, \sum_{\substack{j+k \geq M 
\\ k< M \\ j<M}} 
b_{j,k} \, v_j^{\epsilon} \, v_k^{\epsilon}
\end{equation}

and 
$$
f_m^{\epsilon}(t,x,z)=
\begin{cases}
f^{\epsilon}(t,x,z) & \text{if } m=1 \\
0 & \text{if } 1 <m \leq M
\end{cases}
$$
$$
U_m^{\epsilon}(x,z)=
\begin{cases}
U_1^{\epsilon}(x,z)>0 & \text{if } m=1 \\
0 & \text{if } 1 <m \leq M
\end{cases}
$$
In the system of equations above, 
$\nu_{\epsilon}$ is the outer normal on $\Gamma_{\epsilon}$ with respect
to $\Omega_{\epsilon}$; $f^{\epsilon}(t,x,z) \in C^1([0,T] \times 
\Pi_{\epsilon})$, $c_i(x,z) \in L^{\infty}(\Gamma_{\epsilon}) 
(i=1, \ldots, M)$ and $U_1^{\epsilon}(x,z) \in L^{\infty}(\Pi_{\epsilon})$
are given positive functions.
In (\ref{2.1a}), the unknowns $u_m^{\epsilon}$ and
$v_m^{\epsilon}$ ($1 \leq m < M$) represent the concentration of
$m$-clusters, that is,
clusters consisting of $m$ identical elementary particles (monomers),
while $u_M^{\epsilon}$ and $v_M^{\epsilon}$ take into account aggregations of
more than $M-1$ monomers.
We assume that the only reaction allowing clusters to coalesce to form
larger clusters is a binary coagulation mechanism, while the approach
of two clusters leading to aggregation results only from a diffusion
process.
In particular, we indicate with $D_i$, $\tilde{D}_i$, $d_i$ ($1 \leq i \leq M$)
the positive diffusion coefficients of $i$-clusters in different regions of
the domain and along different directions.
The kinetic coefficients $a_{i,j}$ and $b_{i,j}$ represent a reaction in
which an ($i+j$)-cluster is formed from an $i$-cluster and a $j$-cluster.
Therefore, they can be interpreted as 'coagulation rates' and are
symmetric $a_{i,j}=a_{j,i} >0$, $b_{i,j}=b_{j,i} >0$ ($i,j=1, \ldots, M$).

Our main statement shows that it is possible to homogenize the system of
equations (\ref{2.1a}) as $\epsilon \rightarrow 0$.


\begin{theorem} \label{t6.5}

Let $v_m^{\epsilon}(t,x,z)$ and $u_m^{\epsilon}(t,x,z)$ ($1 \leq m \leq M$)
be a family of weak solutions to the system (\ref{2.1a})
(see Definition \ref{d2.1}).
The sequences $v_m^{\epsilon}$, $\nabla_x v_m^{\epsilon}$, $u_m^{\epsilon}$,
$\epsilon \nabla_x u_m^{\epsilon}$ ($1 \leq m \leq M$), two-scale converge
to: $v_m (t,x,z)$, $\nabla_x v_m (t,x,z)+\nabla_y \tilde{v}_m (t,x,y,z)$,
$u_m(t,x,y,z)$, $\nabla_y u_m(t,x,y,z)$, respectively.
The limiting functions
$v_m \in L^2(0,T; H^1(\Omega)) \bigcap H^1(0,T; L^2(\Omega))$,
$\tilde{v}_m \in L^2([0,T] \times \Omega; H^1_{\sharp}(Y)/\mathbb{R})$,
$u_m \in L^2([0,T] \times \Omega; H^1_{\sharp}(Y)) \bigcap H^1(0,T; L^2(\Omega
\times Y))$
are the unique solutions of the following two-scale homogenized systems.

For $1 \leq m \leq M$ we have:

\begin{equation} \label{hom2.1} 
\begin{split}
&\vert Z \vert \, \frac{\displaystyle \partial v_m}
{\displaystyle \partial t}(t,x,z)
-\text{div}_x \bigg[ d_m \, A \, 
\nabla_x v_m(t,x,z) \bigg]-
d_m \, \vert Z \vert \, \partial_z^2 v_m(t,x,z) \\
&= \displaystyle \int_{\Gamma} c_m(x,z) \, (u_m(t,x,y,z)-
v_m(t,x,z))_+ \, 
d\sigma(y) +\vert Z \vert \,N_m(v) 
\, \, \, 
\; \; \; \; \; \; \; \; \text{ in} \, \, [0,T] \times \Omega
\end{split}
\end{equation}

\begin{equation} \label{hom2.2}
[ A \, \nabla_x v_m (t,x,z)] \cdot \nu=0 \; \; \; \; \; \;  
\; \; \; \; \; \; \; \; \; \; \; \; \; \; \; \; \; \; \; \; \; \; \; \;
\; \; \; \; \; \; \; \; \; \; \; \; \; 
\text{ on} \, \, 
[0,T] \times \Gamma_{L}
\end{equation}

\begin{equation} \label{hom2.3}
\partial_z v_m(t,x,z)=0 \, \, \, \qquad \; \; \; \; \; \; \; \;  \; \; \; 
\; \; \; \; \; \; \; \; \; \; \; \; 
\; \; \; \; \; \; \; \; \; \; \; \; \; \; \; \; \; \; \; \; \; \; \; \;
\text{ on} \, \, [0,T] \times \Gamma_B
\end{equation}

\begin{equation} \label{hom2.4}
v_m (t=0,x,z)=0 \, \, \, \qquad \; \; \; \; \; \; \; \;  \; \; \; \; \; \; \; \;
\; \; \; \; \; \; \; \; \; \; \; \; \; \; \; \; \; \; \; \; \; \; \; \;
\; \; \; \; \; \; \; \; \; \; \; 
\text{ in} \, \, \Omega
\end{equation}
where

$$N_1(v)=-\sum_{j=1}^M b_{1,j} v_1(t, x, z) \, v_j(t, x, z)
$$
$$N_m(v)=\frac{1}{2} \, \sum_{j=1}^{m-1} b_{j,m-j} \, v_j(t, x, z) 
\, v_{m-j}(t, x, z)
-\sum_{j=1}^M b_{m,j} v_m(t, x, z) \, v_j(t, x, z)
\; \; \; (1 < m < M)
$$
$$N_M(v)=\frac{1}{2} \, \sum_{\substack{j+k \geq M 
\\ k< M \\ j<M}} 
b_{j,k} \, v_j(t, x, z) \, v_k(t, x, z)
$$

In (\ref{hom2.1}) and (\ref{hom2.2}), $A$ is a matrix with constant 
coefficients defined by

\begin{equation} \label{hom3.10}
A_{ij}=\displaystyle \int_{Z} (\nabla_y  w_i+ \hat{e}_i) \cdot
(\nabla_y w_j+ \hat{e}_j) \, dy
\end{equation}
with $\hat{e}_i$ being the $i$th unit vector in $\mathbb{R}^3$, and
$(w_i)_{1 \leq i \leq 3}$ the family of solutions of the cell problem

\begin{eqnarray} \label{hom3.11}
\begin{cases}
-\text{div}_y [\nabla_y w_i+ \hat{e}_i]=0 \, \, \, 
\qquad \; \; \; \; \; 
\text{ in} \, \, Z \\
(\nabla_y w_i+\hat{e}_i) \cdot \nu=0 \, \, \, \qquad \; \; \; \; 
\; \; \; \;
\text{ on} \, \, \Gamma \\
y \rightarrow w_i(y) \; \; \; \; \; Y-\text{periodic}
\end{cases}
\end{eqnarray}
appearing in the limiting function

\begin{equation} \label{hom3.9}
{\tilde v}_m(t,x,y,z)=\sum_{i=1}^3 \, w_i (y) \, 
\frac{\displaystyle \partial v_m}{\displaystyle \partial x_i}(t,x,z)
\; \; \; \; (1 \leq m \leq M)
\end{equation}
 
Furthermore, for $1 \leq m \leq M$ we have: 

\begin{equation} \label{hom2.5} 
\begin{split}
&\frac{\displaystyle \partial u_m}{\displaystyle \partial t}(t,x,y,z)-
D_m \, \triangle_y u_m (t,x,y,z)
-\tilde{D}_m \, \partial_z^2 u_m (t,x,y,z) \\
&=L_m(u)+f_m(t,x,y,z) 
\; \; \; \; \; \; \; \; \; \; \; \; \; \; \; \; \; \; \;
t>0, \; \; (x,z) \in \Omega, \; \; y \in X
\end{split}
\end{equation}

\begin{equation} \label{hom2.6}
\begin{split}
&D_m \, \nabla_y  u_m (t,x,y,z) \cdot \nu \\
&=-c_m(x,z) \,
(u_m (t,x,y,z)-v_m (t,x,z))_+  \; \; \; \; \; \; \;
 \;  t>0, \; \;
(x,z) \in \Omega, \; \; y \in \Gamma
\end{split}
\end{equation}

\begin{equation} \label{hom2.7}
\partial_z u_m (t,x,y,z)=0\, \, \, \; \; \; \; \; \; \; \; \; \; \; \;
\; \; \; \; \;
\qquad t>0, \; \;
(x,z) \in \overline{D} \times \{0,L \}, \; \; y \in X
\end{equation}

\begin{equation} \label{hom2.8}
u_m (t=0,x,y,z)=U_m(x,y,z) \, \, \, \qquad \; \; \; \; \; 
\; \; \; \; \; \; \; \; \; \; \; \; \; \; \; \; \; \; \;
\; \; \; \; 
(x,z) \in \Omega, \; \; y \in X
\end{equation}
where

$$L_1(u)=-\sum_{j=1}^M a_{1,j} u_1(t,x,y,z) \, u_j(t,x,y,z)
$$
$$L_m(u)=\frac{1}{2} \, \sum_{j=1}^{m-1} a_{j,m-j} \, 
u_j(t,x,y,z) \, u_{m-j}(t,x,y,z)
-\sum_{j=1}^M a_{m,j} u_m(t,x,y,z) \, u_j(t,x,y,z)
\; \; \; (1 < m < M)
$$
$$L_M(u)= \frac{1}{2} \, \sum_{\substack{j+k \geq M \\ 
k< M \\ j<M}} 
a_{j,k} \, u_j(t,x,y,z) \, u_k(t,x,y,z) 
$$
and

$$
f_m(t,x,y,z)=
\begin{cases}
f(t,x,y,z) & \text{if } m=1 \\
0 & \text{if } 1 <m \leq M
\end{cases}
$$
$$
U_m(x,y,z)=
\begin{cases}
U_1(x,y,z)>0 & \text{if } m=1 \\
0 & \text{if } 1 <m \leq M
\end{cases}
$$

\end{theorem}

\subsection{Motivation}\label{subsec1.1}

The system of equations in (\ref{2.1a}), known under the generic name
of {\sl discrete Smoluchowski equations with diffusion}, is meant to model the
aggregation and diffusion of the pathological tau protein in the
brain, a process associated with the development of a large
variety of cerebral diseases called tauopathies.
Indeed, pathological accumulations of hyperphosphorylated tau protein 
aggregates,
known as neurofibrillary tangles (NFTs), are detected in several
neurodegenerative tauopathies, including Alzheimer's disease (AD)
\cite{GS2, GL, ILGG, PDL}.
Tau is a highly soluble, natively unfolded protein which is predominantly
located in the axons of neurons of the central nervous system.
Here, its physiological function is to support assembly and stabilization
of axonal microtubules.
Under pathological conditions, tau can assume abnormal conformations,
due to two transformations: hyperphosphorylation and misfolding.
In particular, hyperphosphorylation has a negative impact on the
biological function of tau proteins since it inhibits the binding to
microtubules, compromising their stabilization and axonal transport and
promotes self-aggregation.
Thus, misfolded tau monomers constitute the building unit for the formation of
oligomers which
in turn lead to highly structured and insoluble fibrils.
For many years, cell autonomous mechanisms were believed to be responsible
for the evolution of neurodegenerative diseases, implying that the same
aggregation events occur independently in different brain cells.
However, accumulating evidence now demonstrates that the progression of tau
pathology reflects cell-to-cell propagation of the disease, achieved through
the release of tau into the extracellular space and the uptake by surrounding
healthy neurons \cite{FSJGK, Yam}.
Extracellular tau then seeds physiological tau in the recipient cells to
propagate the pathological process through neural pathways made by
bundles of axons called tracts.
This mechanism, often referred to as 'prion-like' propagation of tau
pathology, has led to the idea that extracellular tau could be a novel
therapeutic target to halt the spread of the disease \cite{PLBMMG}.
For this reason, in the present work, we focus on a model which describes
in a simplified way the tau diffusion along tracts starting from the
microscopic release and uptake mechanism.

In this context, the set $\Pi_{\epsilon}$ represents a bundle of axons
in the white matter (a tract), while the
domain $\Omega_{\epsilon}$ indicates the extracellular region filled by
cerebrospinal fluid.
The variables $u_m^{\epsilon} \geq 0$ and
$v_m^{\epsilon} \geq 0$ refer to the concentration of $m$-clusters of
hyperphosphorylated tau spreading within the neuronal axons (represented
by the long cylindrical cavities) and in extracellular space, respectively.
Concerning the diffusion process, the choice made in our model to define
two coefficients reflects the biological properties of the tau protein
that diffuses differently in extracellular space and in the axon \cite{GS2}.
Moreover, within the neuronal axon, we further distinguish between two
different diffusion coefficients, since tau propagates preferentially along
the $z$-axis of the cylinder that represents the axon.
The initial condition $U_1^{\epsilon}(x,z)$, given at $t=0$ by the
concentration of monomers diffusing within the axons, represents the amount
of endogenous misfolded tau protein,
while the source term $f^{\epsilon}(t,x,z)$, in the system
(\ref{2.1a}) for $m=1$, indicates the production of 
hyperphosphorylated tau monomers.

\subsection{Background of this work}\label{subsec1.2}

There is a large literature related to the use of the Smoluchowski equation
in various physical contexts (e.g., 
\cite{Wrz, LM, LM2, FL, Rez, AM}),
but only a few works
concerning its application in the biomedical field \cite{CIA, BFRT}.
Recently, the important role of the Smoluchowski equations in modelling
at different scales the evolution of neurodegenerative diseases, such as
AD, has been investigated in 
\cite{MP, AFMT, BFMTT, BFRT, RTGCCRLF, FHT, FL1, FHL} .
In fact, the present work is part of a broader research effort carried 
on by various 
groups of researchers with diverse collaborations on mathematical models of 
the progression of AD, and represents an initial bridge 
between microscopic models of tau diffusion developed in biology 
and macroscopic mathematical 
models based on graph theory.
To this end, the homogenization technique, introduced by the mathematicians
in the seventies to carry out a sort of averaging process on the solutions
of partial differential equations with rapidly varying coefficients or
describing media with microstructures, has been applied
\cite{CD, CP, dalmaso, GAGP, Jos, GNRP, CK}.

It is nowadays generally accepted that tau protein, in synergetic
combination with another protein, the so-called $\beta$-amyloid peptide, 
plays a key role in the
development of AD (see \cite{BH}).
We refer for instance to \cite{BFRT} for a discussion on macroscopic
mathematical modeling of this interaction. Here, 
our interest is focused on the tau protein 
that diffuses through the neural pathway, whereas we ignore
deliberately the action of the $\beta$-amyloid.

Unlike the approach proposed for instance in \cite{RTGCCRLF}, \cite{BFRT}, 
where the modeling of
 tau coa\-gu\-la\-tion-diffusion processes has been carried out
on a large scale, that is the scale of the neural network characterized
by the connectivity of different regions through bundles of axons (tracts),
in this paper only the mesoscopic dynamics within a portion of tract has
been investigated.
Starting from the derivation of model equations valid at the microscale,
by using the so-called {\sl two-scale homogenization} technique 
we have proved that the solution
two-scale converges to the solution of a macromodel asymptotically consistent
with the original one.
The notion of two-scale homogenization has been first introduced by 
Nguetseng \cite{Ngu} and
Allaire \cite{All} in the deterministic periodic setting, and later
generalized to the stochastic framework by Zhikov and Piatnitsky \cite{ZP}.
Unlike other homogenization techniques (see \cite{dalmaso}, \cite{CD} for a
review),
the two-scale convergence method is self-contained in that, in a single
step, one can derive the homogenized equations and prove the convergence
of the sequence of solutions to the problem at hand.

To stress the novelty of the present paper,
it is worth noting that two-scale homogenization techniques have been 
already used by the authors to pass from microscopic to macroscopic
model of the diffusion of toxic proteins in the cerebral parenchyma affected
by AD, but in utterly different biological perspectives and
geometries.
Indeed, in \cite{FL1} and \cite{FHL}, the authors aimed to describe
production, aggregation and diffusion of $\beta$-amyloid peptide 
in the cerebral
tissue (macroscopic scale), a process associated with the development of
AD, starting from the derivation of a model at the single
neuron level (microscopic scale). Thus, the present paper differs from
 our previous works  both for the biological meaning
 and, consequently, 
for the geometry of the problem and the
boundary conditions introduced in order to take into account the
peculiarities of tau protein propagation.

There is a large literature devoted to the study of transmission boundary
conditions somehow akin to those of \eqref{2.1a}, especially in the 
framework of porous media \cite{JMN}.
The main results in this respect, which are relevant to our work, can be
found in \cite{HJM, NEUSS, KAM}.
In particular, in \cite{HJM}, it is assumed that the porous medium is
composed of periodically arranged cubic cells of size $\epsilon$, split up
into a solid part (a ball surrounded by semi-permeable membranes) and a
fluid part.
In this setting, the diffusion and reactions of chemical species in the
fluid and in the solid part are studied, while transmission boundary
conditions are imposed on the interface.
The same geometry is considered also in \cite{KAM}, where deposition effects
under the influence of thermal gradients are analyzed.
The model takes into account the motion of populations of colloidal particles
dissolved in the water interacting together via Smoluchowski coagulation
terms.
The colloidal matter cannot penetrate the solid grain boundary, but it
deposits there. This process is again described by transmission boundary
conditions.
A domain decomposed into long cylindrical cavities periodically distributed
has been defined in \cite{NEUSS} to model a porous medium consisting of a
fluid part and solid bars.
Chemical substances, dissolved in the fluid, are transported by diffusion
and adsorbed on the surface of the bars (through transmission boundary
conditions) where chemical reactions take place.

\subsection{Outline of the paper}\label{subsec1.3}

The paper is organized as follows.
In Section \ref{sec2} we prove the existence of weak solutions to the
system of Smoluchowski-type equations (\ref{2.1a}), while
Section \ref{sec3} is devoted to the proof of their positivity and
boundedness.
A priori estimates on the derivatives of the solutions are also obtained and
reported at the end of Section \ref{sec3}.
In Section \ref{sec4} we study the convergence of the homogenization process
and prove our main result concerning the two-scale limit of the solutions
to the set of equations (\ref{2.1a}).
Finally, in the Appendices we recall some basic Theorems related to
functional analysis and on the two-scale convergence method.

\section{Existence of solutions}\label{sec2}

Let us consider the following truncation of the nonlinear terms
in the Smoluchowski-type equations (\ref{2.1a}) for $1 \leq i \leq M$ 
\cite{KAM}:

\begin{equation} \label{2.4}
L_i^{\tilde M}(u^{\epsilon}):=L_i(\sigma_{\tilde M}(u_1^{\epsilon}),
\sigma_{\tilde M}(u_2^{\epsilon}), \dots, \sigma_{\tilde M}(u_M^{\epsilon}))
\end{equation}

\begin{equation} \label{2.5}
N_i^{\tilde M}(v^{\epsilon}):=N_i(\sigma_{\tilde M}(v_1^{\epsilon}),
\sigma_{\tilde M}(v_2^{\epsilon}), \dots, \sigma_{\tilde M}(v_M^{\epsilon}))
\end{equation}
where

\begin{equation} \label{2.6}
\sigma_{\tilde M}(s):=\begin{cases}
0, \; \; \; \; s<0 \\
s, \; \; \; \; s \in [0, {\tilde M}] \\
{\tilde M}, \; \; \; \; s>{\tilde M}
\end{cases}
\end{equation}
with ${\tilde M}>0$ being a fixed threshold.
If ${\tilde M}$ is large enough, the estimates derived later in this paper
will give bounds that will remain below ${\tilde M}$.
This means that the results obtained in the following hold also for the
uncutted coagulation terms.

\begin{definition} \label{d2.1}
The functions $u_i^{\epsilon} \in H^1 ([0,T]; L^2(\Pi_{\epsilon})) 
\cap L^{\infty} ([0,T]; H^1(\Pi_{\epsilon}))$
and  \\
$v_i^{\epsilon} \in H^1 ([0,T]; L^2(\Omega_{\epsilon})) 
\cap L^{\infty} ([0,T]; H^1(\Omega_{\epsilon}))$
($1\leq i \leq M$)
are solutions to
the problem (\ref{2.1a}) if the following relations hold, a.e.
in $[0,T]$ and for a fixed value of $\epsilon >0$.

If $1 \leq i \leq M$:

\begin{eqnarray} \label{2.7a} \nonumber
&\displaystyle \int_{\Pi_{\epsilon}} \partial_t u_i^{\epsilon} \, \psi_i
\, dx \, dz+
\epsilon^2 \displaystyle \int_{\Pi_{\epsilon}} D_i 
\nabla_x u_i^{\epsilon} \cdot \nabla_x \psi_i \, dx \, dz+
\displaystyle \int_{\Pi_{\epsilon}} \tilde{D}_i \, \partial_z u_i^{\epsilon} 
\cdot \partial_z \psi_i \, dx \, dz\\ \nonumber
&+\epsilon \displaystyle \int_{\Gamma_{\epsilon}} c_i(x,z) 
(u_i^{\epsilon}-v_i^{\epsilon})_{+} \, \psi_i \, d\sigma_{\epsilon}  
=\displaystyle \int_{\Pi_{\epsilon}} L_i^{\tilde M}(u^{\epsilon}) \, \psi_i
\, dx \, dz \\
&+\displaystyle \int_{\Pi_{\epsilon}} f_i^{\epsilon}(t,x,z) \, \psi_i
\, dx \, dz
\end{eqnarray}
for all $\psi_i \in H^1(\Pi_{\epsilon})$, and

\begin{eqnarray} \label{2.8a}  \nonumber
&\displaystyle \int_{\Omega_{\epsilon}} \partial_t v_i^{\epsilon} \, \phi_i
\, dx \, dz+
\displaystyle \int_{\Omega_{\epsilon}} d_i 
\nabla v_i^{\epsilon} \cdot \nabla \phi_i \, dx \, dz-
\epsilon \displaystyle \int_{\Gamma_{\epsilon}} c_i(x,z) 
(u_i^{\epsilon}-v_i^{\epsilon})_{+} \, \phi_i  \, d\sigma_{\epsilon}\\  
&=\displaystyle \int_{\Omega_{\epsilon}} N_i^{\tilde M}(v^{\epsilon}) \, 
\phi_i \, dx \, dz
\end{eqnarray}
for all $\phi_i \in H^1(\Omega_{\epsilon})$, along with the initial conditions
$u_i^{\epsilon}(0,x,z)$  and $v_i^{\epsilon}(0,x,z)$.

\end{definition}

\begin{remark}
In Eqs. (\ref{2.7a})-(\ref{2.8a}) the integrals over the boundary
$\Gamma_{\epsilon}$ are well defined thanks to the regularity of the
functions $u_i^{\epsilon}$ and $v_i^{\epsilon}$, as stated in
Definition \ref{d2.1}, and the existence of the interpolation-trace
inequality given by Eq.(\ref{A.7}) in Appendix A.
With an abuse of notation, in Eqs. (\ref{2.7a})-(\ref{2.8a}) we have indicated
with $u_i^{\epsilon}$ and $v_i^{\epsilon}$ also the trace of these functions.

\end{remark}

\begin{lemma} \label{l4.1}
For a given small $\epsilon >0$,
the system (\ref{2.1a}) has a solution \\
$u_i^{\epsilon} \in  H^1 ([0,T]; L^2(\Pi_{\epsilon})) \cap L^{\infty} ([0,T]; 
H^1(\Pi_{\epsilon}))$
and
$v_i^{\epsilon} \in H^1 ([0,T]; L^2(\Omega_{\epsilon})) \cap L^{\infty} ([0,T]; 
H^1(\Omega_{\epsilon}))$ ($1 \leq i \leq M$)
in the sense of the Definition \ref{d2.1}.
\end{lemma}

\begin{proof}

Let \{$\xi_j$\} be an orthonormal basis of
$H^1 (\Pi_{\epsilon})$ and \{$\eta_j$\} of $H^1 (\Omega_{\epsilon})$.
We denote by $u_{i,n}^{\epsilon}$ and $v_{i,n}^{\epsilon}$ ($1 \leq i \leq M$)
the Galerkin
approximations of $u_i^{\epsilon}$ and $v_i^{\epsilon}$, respectively, that is

\begin{equation} \label{4.1}
u_{i,n}^{\epsilon} (t,x,z):= \sum_{j=1}^n \alpha_{i,j}^n (t) \, \xi_j (x,z)
\end{equation}
for all $t \in [0,T]$, $(x,z) \in \Pi_{\epsilon}$ and

\begin{equation} \label{4.2}
v_{i,n}^{\epsilon} (t,x,z):= \sum_{j=1}^n \beta_{i,j}^n (t) \, \eta_j (x,z)
\end{equation}
for all $t \in [0,T]$, $(x,z) \in \Omega_{\epsilon}$.
Let now $i=1$.

Since \{$\xi_j$\} is an orthonormal basis of $H^1 (\Pi_{\epsilon})$,
for each $n \in \mathbb {N}$, there exists

\begin{equation} \label{4.3}
u_{1,n}^{\epsilon,0} (x,z)= \sum_{j=1}^n \alpha_{1,j}^{0,n} \, \xi_j (x,z)
\end{equation}
so that $u_{1,n}^{\epsilon,0} \rightarrow U_1^{\epsilon}(x,z)$ in $H^1 
(\Pi_{\epsilon})$
as $n \rightarrow \infty$.

Likewise, for each $n \in \mathbb {N}$, there exists

\begin{equation} \label{4.4}
v_{1,n}^{\epsilon,0} (x,z)= \sum_{j=1}^n \beta_{1,j}^{0,n} \, \eta_j (x,z)
\end{equation}
so that $v_{1,n}^{\epsilon,0} \rightarrow 0$ in $H^1 (\Omega_{\epsilon})$
as $n \rightarrow \infty$.

To derive the coefficients of the Galerkin approximations, we impose that
the functions $u_{1,n}^{\epsilon}$ and $v_{1,n}^{\epsilon}$
satisfy Eqs. (\ref{2.7a}) and (\ref{2.8a}).
Therefore, for all $\psi_1 \in {\mathrm span}\{\xi_j\}_{j=1}^n $,
Eq. (\ref{2.7a}) can be written as:

\begin{eqnarray} \label{4.5} \nonumber
&\displaystyle \int_{\Pi_{\epsilon}} \partial_t u_{1,n}^{\epsilon} \, \psi_1
\, dx \, dz+
\epsilon^2 \, \displaystyle \int_{\Pi_{\epsilon}} D_1 \,
\nabla_x u_{1,n}^{\epsilon} \cdot \nabla_x \psi_1 \, dx \, dz+
\displaystyle \int_{\Pi_{\epsilon}} \tilde{D}_1 \, \partial_z u_{1,n}^{\epsilon}
\cdot \partial_z \psi_1 \, dx \, dz\\ \nonumber
&+ \epsilon \,
\displaystyle \int_{\Gamma_{\epsilon}} c_1(x,z) \,
(u_{1,n}^{\epsilon}-v_{1,n}^{\epsilon})_{+} \, \psi_1 \, d\sigma_{\epsilon}
=\displaystyle \int_{\Pi_{\epsilon}} L_1^{\tilde M} (u_{i,n}^{\epsilon}) \,
\psi_1 \, dx \, dz \\
&+\displaystyle \int_{\Pi_{\epsilon}} f^{\epsilon}(t,x,z) 
\, \psi_1 \, dx \, dz
\end{eqnarray}
By testing Eq. (\ref{4.5}) with $\psi_1=\xi_k$, we get

\begin{eqnarray} \label{4.6} \nonumber
&\displaystyle \int_{\Pi_{\epsilon}}  \partial_t \bigg[ 
\sum_{j=1}^n \alpha_{1,j}^n (t) \, \xi_j (x,z) \bigg] \, \xi_k \, dx \, dz+
\epsilon^2 \, \displaystyle \int_{\Pi_{\epsilon}} D_1 \,
\bigg[ \sum_{j=1}^n \alpha_{1,j}^n (t) \, \nabla_x \xi_j (x,z) \bigg] \cdot
\nabla_x \xi_k \, dx \, dz \\ \nonumber
&+\displaystyle \int_{\Pi_{\epsilon}} \tilde{D}_1 \, 
\bigg[ \sum_{j=1}^n \alpha_{1,j}^n (t) \, \partial_z \xi_j (x,z) \bigg] \cdot
\partial_z \xi_k \, dx \, dz\\ \nonumber 
&+\epsilon \, \displaystyle \int_{\Gamma_{\epsilon}} c_1(x,z) \,
\bigg[ \sum_{j=1}^n \alpha_{1,j}^n (t) \, \xi_j (x,z)-
\sum_{j=1}^n \beta_{1,j}^n (t) \, \eta_j (x,z) \bigg]_{+} \, \xi_k 
\, d\sigma_{\epsilon}\\ \nonumber
&=-\displaystyle \int_{\Pi_{\epsilon}} \xi_k \, 
\sum_{a=1}^M a_{1,a} \, \sigma_{\tilde M} \bigg[
\sum_{b=1}^n \alpha_{1,b}^n (t) \, \xi_b (x,z) \bigg] \,
\sigma_{\tilde M} \bigg[ \sum_{c=1}^n \alpha_{a,c}^n (t) \, \xi_c (x,z) \bigg] 
\, dx \, dz
\\&+\displaystyle \int_{\Pi_{\epsilon}} f^{\epsilon}(t,x,z) \, \xi_k \, 
dx \, dz.
\end{eqnarray}
Hence, for $k \in \{1, \ldots ,n \}$:

\begin{eqnarray} \label{4.7} \nonumber
&\partial_t \alpha_{1,k}^n (t)+\sum_{j=1}^n A_{1jk} \, \alpha_{1,j}^n (t)=
-F_{1,k} (\alpha_1^n (t), \beta_1^n (t)) \\ \nonumber
&-\displaystyle \int_{\Pi_{\epsilon}} \xi_k \,
\sum_{a=1}^M a_{1,a} \, \sigma_{\tilde M} \bigg[
\sum_{b=1}^n \alpha_{1,b}^n (t) \, \xi_b (x,z) \bigg] \,
\sigma_{\tilde M} \bigg[ \sum_{c=1}^n \alpha_{a,c}^n (t) \, \xi_c (x,z) \bigg]
\, dx \, dz
\\&+\displaystyle \int_{\Pi_{\epsilon}} f^{\epsilon}(t,x,z) \, \xi_k \, 
dx \, dz
\end{eqnarray}
where the coefficients $A_{1jk}$ are defined by

\begin{equation} \label{4.8}
A_{1jk}:=\epsilon^2 \displaystyle \int_{\Pi_{\epsilon}} D_1 \,
\nabla_x \xi_j \cdot \nabla_x \xi_k \, dx\, dz+
\displaystyle \int_{\Pi_{\epsilon}} \tilde{D}_1 \, \partial_z \xi_j \cdot
\partial_z \xi_k \, dx\, dz  
\end{equation}
and

\begin{equation} \label{4.9}
F_{1,k} (\alpha_1^n (t), \beta_1^n (t)):=\epsilon 
\displaystyle \int_{\Gamma_{\epsilon}} c_1(x,z) \, 
(u_{1,n}^{\epsilon}-v_{1,n}^{\epsilon})_{+} \, \xi_k \, \, d\sigma_{\epsilon}
\end{equation}

We choose now $v_{1,n}^{\epsilon}$ satisfying Eq. (\ref{2.8a}) for all
$\phi_1 \in {\mathrm span}\{\eta_j\}_{j=1}^n $, i.e.:

\begin{eqnarray} \label{4.10} \nonumber
&\displaystyle \int_{\Omega_{\epsilon}} \partial_t v_{1,n}^{\epsilon} \, 
\phi_1 \, dx\, dz+
\displaystyle \int_{\Omega_{\epsilon}} d_1 \,
\nabla v_{1,n}^{\epsilon} \cdot \nabla \phi_1 \, dx\, dz \\
&- \epsilon \,
\displaystyle \int_{\Gamma_{\epsilon}} c_1(x,z) \,
(u_{1,n}^{\epsilon}-v_{1,n}^{\epsilon})_{+} \, \phi_1 \, d\sigma_{\epsilon}
=\displaystyle \int_{\Omega_{\epsilon}} N_1^{\tilde M} (v_{i,n}^{\epsilon}) \,
\phi_1 \, dx\, dz
\end{eqnarray}
By testing Eq. (\ref{4.10}) with $\phi_1=\eta_k$, we get

\begin{eqnarray} \label{4.11} \nonumber
&\displaystyle \int_{\Omega_{\epsilon}}  \partial_t \bigg[ 
\sum_{j=1}^n \beta_{1,j}^n (t) \, \eta_j (x,z) \bigg] \, \eta_k \, dx\, dz+
\displaystyle \int_{\Omega_{\epsilon}} d_1 \,
\bigg[ \sum_{j=1}^n \beta_{1,j}^n (t) \, \nabla \eta_j (x,z) \bigg] \cdot
\nabla \eta_k \, dx\, dz \\ \nonumber
&-\epsilon \, \displaystyle \int_{\Gamma_{\epsilon}} c_1(x,z) \,
\bigg[ \sum_{j=1}^n \alpha_{1,j}^n (t) \, \xi_j (x,z)-
\sum_{j=1}^n \beta_{1,j}^n (t) \, \eta_j (x,z) \bigg]_{+} \, \eta_k 
\, d\sigma_{\epsilon}\\
&=-\displaystyle \int_{\Omega_{\epsilon}} \eta_k \, 
\sum_{a=1}^M b_{1,a} \, \sigma_{\tilde M} \bigg[
\sum_{b=1}^n \beta_{1,b}^n (t) \, \eta_b (x,z) \bigg] \,
\sigma_{\tilde M} \bigg[ \sum_{c=1}^n \beta_{a,c}^n (t) \, \eta_c (x,z) \bigg]
\, dx\, dz
\end{eqnarray}
Hence, for $k \in \{1, \ldots ,n \}$:

\begin{eqnarray} \label{4.12} \nonumber
&\partial_t \beta_{1,k}^n (t)+\sum_{j=1}^n B_{1jk} \, \beta_{1,j}^n (t)=
G_{1,k} (\alpha_1^n (t), \beta_1^n (t)) \\
&-\displaystyle \int_{\Omega_{\epsilon}} \eta_k \,
\sum_{a=1}^M b_{1,a} \, \sigma_{\tilde M} \bigg[
\sum_{b=1}^n \beta_{1,b}^n (t) \, \eta_b (x,z) \bigg] \,
\sigma_{\tilde M} \bigg[ \sum_{c=1}^n \beta_{a,c}^n (t) \, \eta_c (x,z) \bigg]
\, dx\, dz
\end{eqnarray}
where the coefficients $B_{1jk}$ are defined by

\begin{equation} \label{4.13}
B_{1jk}:=\displaystyle \int_{\Omega_{\epsilon}} d_1 \,
\nabla \eta_j \cdot \nabla \eta_k \, dx \, dz
\end{equation}
and

\begin{equation} \label{4.14}
G_{1,k} (\alpha_1^n (t), \beta_1^n (t)):=\epsilon 
\displaystyle \int_{\Gamma_{\epsilon}} c_1(x,z) \, 
(u_{1,n}^{\epsilon}-v_{1,n}^{\epsilon})_{+} \, \eta_k \, \, d\sigma_{\epsilon}
\end{equation}

Eqs.(\ref{4.7}) and (\ref{4.12}) represent a system of $2 n$ ordinary
differential
equations for the coefficients $\alpha_1^n=(\alpha_{1,k}^n)_{k=1, \ldots,n}$,
$\beta_1^n=(\beta_{1,k}^n)_{k=1, \ldots,n}$.
By taking into account formulas (\ref{4.3}) and (\ref{4.4}), we assume as
initial conditions:

\begin{equation} \label{4.14a}
\alpha_{1,k}^{n}(0)=\alpha_{1,k}^{0,n}; \; \; \; \; 
\beta_{1,k}^{n}(0)=\beta_{1,k}^{0,n}.
\end{equation}

Since the left-hand side of this system of ordinary differential equations
is linear and the right-hand side has a sublinear growth (thanks to the
Lipschitz property of the functions),
one can conclude that the
Cauchy problem (\ref{4.7}), (\ref{4.12}), (\ref{4.14a}) has a unique solution
extended to the whole interval $[0,T]$
(see the generalization of the Picard-Lindel\"of theorem in \cite{Hart},
Theorem 5.1, p. 156).
Moreover, $\alpha_1^n=(\alpha_{1,k}^n)_{k=1, \ldots,n}(t)$,
$\beta_1^n=(\beta_{1,k}^n)_{k=1, \ldots,n}(t) \in H^1([0,T])$
for $t \in [0,T]$.
We prove in the following the global Lipschitz property of $F_{1,k}$.
The proof of the Lipschitz continuity of $G_{1,k}$ is similar.

Let $(u_{1,n}^{\epsilon}, v_{1,n}^{\epsilon})$ and
$(\tilde {u}_{1,n}^{\epsilon}, \tilde {v}_{1,n}^{\epsilon})$ be of the form
(\ref{4.1}) and (\ref{4.2}), with coefficients
($\alpha_1^n(t)=(\alpha_{1,k}^n)_{k=1, \ldots, n}$,
$\beta_1^n(t)=(\beta_{1,k}^n)_{k=1, \ldots, n}$)  and
($\tilde {\alpha}_{1}^n(t)=(\tilde {\alpha}_{1,k}^n)_{k=1, \ldots, n}$,
$\tilde {\beta}_{1}^n(t)=(\tilde {\beta}_{1,k}^n)_{k=1, \ldots, n}$),
respectively.

One has:

\begin{eqnarray} \label{4.15} \nonumber
&F_{1,k} (\alpha_1^n (t), \beta_1^n (t))-F_{1,k} (\tilde {\alpha}_1^n (t), 
\tilde {\beta}_1^n (t)) \\  \nonumber
&=\epsilon \displaystyle \int_{\Gamma_{\epsilon}} c_1(x,z) \,
\bigg[ (u_{1,n}^{\epsilon}-v_{1,n}^{\epsilon})_{+}-
(\tilde {u}_{1,n}^{\epsilon}-\tilde {v}_{1,n}^{\epsilon})_{+} \bigg] 
\xi_k (x,z) \,
\, d\sigma_{\epsilon}\\ \nonumber
& \leq \epsilon \, C_1 \displaystyle \int_{\Gamma_{\epsilon}} c_1(x,z) \,
\vert (u_{1,n}^{\epsilon}-\tilde {u}_{1,n}^{\epsilon})-
(v_{1,n}^{\epsilon}-\tilde {v}_{1,n}^{\epsilon}) \vert \, \vert \xi_k (x,z) 
\vert
\, d\sigma_{\epsilon} \\ \nonumber
&=\epsilon  \, C_1 \displaystyle \int_{\Gamma_{\epsilon}} c_1(x,z) 
\bigg \vert \sum_{j=1}^n (\alpha_{1,j}^n (t)-\tilde {\alpha}_{1,j}^n (t))
\xi_j (x,z) \\ \nonumber
&-\sum_{j=1}^n (\beta_{1,j}^n (t)-\tilde {\beta}_{1,j}^n (t)) 
\eta_j (x,z)
\bigg \vert  \vert \xi_k (x,z) \vert \, \, d\sigma_{\epsilon} \\ \nonumber
&\leq \epsilon \, C_1 \sum_{j=1}^n \vert \alpha_{1,j}^n (t)-
\tilde {\alpha}_{1,j}^n (t) \vert \displaystyle \int_{\Gamma_{\epsilon}} 
c_1(x,z) \, \vert \xi_j (x,z) \vert \, \vert \xi_k (x,z) \vert \, \, 
d\sigma_{\epsilon}
\\ \nonumber
&+\epsilon \, C_1 \sum_{j=1}^n \vert \beta_{1,j}^n (t)-
\tilde {\beta}_{1,j}^n (t) \vert \displaystyle \int_{\Gamma_{\epsilon}}
c_1(x,z) \, \vert \eta_j (x,z) \vert \, \vert \xi_k (x,z) \vert 
\, d\sigma_{\epsilon} \\ \nonumber
\nonumber
& \leq \epsilon \,  C_1 \max \{c_{jk}\} \sum_{j=1}^n \vert \alpha_{1,j}^n (t)-
\tilde {\alpha}_{1,j}^n (t) \vert \\
&+\epsilon \, C_1 \max \{d_{jk}\}
\sum_{j=1}^n \vert \beta_{1,j}^n (t)-\tilde {\beta}_{1,j}^n (t) \vert
\end{eqnarray}
where the coefficients ($c_{jk}$) and ($d_{jk}$) are given by

\begin{equation} \label{4.16}
c_{jk}:=\displaystyle \int_{\Gamma_{\epsilon}} c_1(x,z) \vert \xi_j (x,z) \vert 
\, \vert \xi_k (x,z) \vert \, \, d\sigma_{\epsilon}
\end{equation}

\begin{equation} \label{4.17}
d_{jk}:=\displaystyle \int_{\Gamma_{\epsilon}} c_1(x,z) \vert \eta_j (x,z) 
\vert \, \vert \xi_k (x,z) \vert \, \, d\sigma_{\epsilon}
\end{equation}
for $j,k=1, \ldots ,n$.
Hence, we get

\begin{equation} \label{4.18}
\vert F_{1,k} (\alpha_1^n, \beta_1^n)-F_{1,k} (\tilde {\alpha}_1^n, 
\tilde {\beta}_1^n) \vert \leq
\epsilon \, C_2(n) \bigg[ \vert \alpha_1^n-\tilde {\alpha}_1^n \vert+
\vert \beta_1^n-\tilde {\beta}_1^n \vert \bigg].
\end{equation}
The same conclusions can be drawn also when $1 < i \leq M$ by applying
exactly the arguments considered above.

\noindent

\textbf{Uniform estimates}

\noindent

Let us now prove uniform estimates in $n$ for $u_{i,n}^{\epsilon}$ and
$v_{i,n}^{\epsilon}$ ($1 \leq i \leq M$).
In the case $i=1$, we take in Eq. (\ref{4.5})
$\psi_1=u_{1,n}^{\epsilon}(t, \cdot)$ as test
function:

\begin{eqnarray} \label{4.18a}\nonumber
&\frac{\displaystyle 1}{\displaystyle 2} 
  \displaystyle \int_{\Pi_{\epsilon}} \partial_t (u_{1,n}^{\epsilon})^2  
\, dx \, dz
+\epsilon^2 \displaystyle \int_{\Pi_{\epsilon}} D_1
\vert \nabla_x u_{1,n}^{\epsilon} \vert^2 \, dx \, dz \\ \nonumber
&+\displaystyle \int_{\Pi_{\epsilon}} \tilde{D}_1 \vert \partial_z 
u_{1,n}^{\epsilon} \vert^2 \, dx \, dz
+ \epsilon \displaystyle \int_{\Gamma_{\epsilon}} c_1(x,z) 
(u_{1,n}^{\epsilon}-v_{1,n}^{\epsilon})_{+} \, u_{1,n}^{\epsilon}
\, d\sigma_{\epsilon} \\
&=-\displaystyle \int_{\Pi_{\epsilon}} \bigg[ \sum_{j=1}^{M} a_{1,j} \,
\sigma_{\tilde M} (u_{1,n}^{\epsilon})  
\sigma_{\tilde M} (u_{j,n}^{\epsilon}) 
\bigg] \, u_{1,n}^{\epsilon} \, dx \, dz
+\displaystyle \int_{\Pi_{\epsilon}} f^{\epsilon}(t,x,z) \, u_{1,n}^{\epsilon}
\, dx \, dz
\end{eqnarray}
Since the first term on the right-hand side is always negative due to the
truncation of the coagulation terms (see Eq. (\ref{2.6})), one obtains:

\begin{eqnarray} \label{4.19} \nonumber
&\frac{\displaystyle 1}{\displaystyle 2} \partial_t \Vert u_{1,n}^{\epsilon}
\Vert^2_{L^2(\Pi_{\epsilon})}+\epsilon^2 \, D_1 \Vert
\nabla_x u_{1,n}^{\epsilon} \Vert^2_{L^2(\Pi_{\epsilon})} 
+\tilde{D}_1 \Vert \partial_z u_{1,n}^{\epsilon} \Vert^2_{L^2(\Pi_{\epsilon})}
\\ \nonumber
&\leq
-\epsilon \displaystyle \int_{\Gamma_{\epsilon}} c_1(x,z) (u_{1,n}^{\epsilon}
-v_{1,n}^{\epsilon})_{+} \, u_{1,n}^{\epsilon} \, d\sigma_{\epsilon} 
+\frac{\displaystyle 1}{\displaystyle 2} \Vert f^{\epsilon}(t,x,z)
\Vert^2_{L^2(\Pi_{\epsilon})} \\
&+\frac{\displaystyle 1}{\displaystyle 2}
\Vert u_{1,n}^{\epsilon}\Vert^2_{L^2(\Pi_{\epsilon})}
\end{eqnarray}
where we have applied the H\"older inequality to the last term on the
right-hand side of Eq. (\ref{4.18a}).
Since the function $f^{\epsilon}(t,x,z)$ is bounded in
$L^2([0,T] \times \Pi_{\epsilon})$,
Eq. (\ref{4.19}) reads:

\begin{eqnarray} \label{4.19a} \nonumber
&\frac{\displaystyle 1}{\displaystyle 2} \partial_t \Vert u_{1,n}^{\epsilon}
\Vert^2_{L^2(\Pi_{\epsilon})}+\epsilon^2 \, D_1 \Vert
\nabla_x u_{1,n}^{\epsilon} \Vert^2_{L^2(\Pi_{\epsilon})} 
+\tilde{D}_1 \Vert \partial_z u_{1,n}^{\epsilon} \Vert^2_{L^2(\Pi_{\epsilon})}
\\
&\leq
-\epsilon \displaystyle \int_{\Gamma_{\epsilon}} c_1(x,z) (u_{1,n}^{\epsilon}
-v_{1,n}^{\epsilon})_{+} \, u_{1,n}^{\epsilon} \, d\sigma_{\epsilon}
+C_f+\frac{\displaystyle 1}{\displaystyle 2}
\Vert u_{1,n}^{\epsilon} \Vert^2_{L^2(\Pi_{\epsilon})}
\end{eqnarray}
where $C_f$ is a positive constant.
By testing now Eq. (\ref{4.10}) with $\phi_1=v_{1,n}^{\epsilon}(t, \cdot)$,
we get

\begin{eqnarray} \nonumber
&\frac{\displaystyle 1}{\displaystyle 2} 
  \displaystyle \int_{\Omega_{\epsilon}} \partial_t (v_{1,n}^{\epsilon})^2  
\, dx \, dz
+\displaystyle \int_{\Omega_{\epsilon}} d_1 
\vert \nabla v_{1,n}^{\epsilon} \vert^2 \, dx \, dz
-\epsilon \displaystyle \int_{\Gamma_{\epsilon}} c_1(x,z) 
(u_{1,n}^{\epsilon}-v_{1,n}^{\epsilon})_{+} \, v_{1,n}^{\epsilon}
\, d\sigma_{\epsilon}
\\  
&=-\displaystyle \int_{\Omega_{\epsilon}} \bigg[ \sum_{j=1}^{M} b_{1,j} \,
\sigma_{\tilde M} (v_{1,n}^{\epsilon}) \, 
\sigma_{\tilde M} (v_{j,n}^{\epsilon}) 
\bigg] \, v_{1,n}^{\epsilon} \, dx \, dz
\end{eqnarray}
Since again the term on the right-hand side is negative, we conclude:

\begin{equation} \label{4.20}
\frac{\displaystyle 1}{\displaystyle 2} \partial_t \Vert v_{1,n}^{\epsilon}
\Vert^2_{L^2(\Omega_{\epsilon})}+d_1 \Vert
\nabla v_{1,n}^{\epsilon} \Vert^2_{L^2(\Omega_{\epsilon})} \leq
\epsilon \displaystyle \int_{\Gamma_{\epsilon}} c_1(x,z) (u_{1,n}^{\epsilon}
-v_{1,n}^{\epsilon})_{+} \, v_{1,n}^{\epsilon} \, d\sigma_{\epsilon}
\end{equation}
Adding the inequalities (\ref{4.19a}) and (\ref{4.20}), it follows that

\begin{eqnarray} \label{4.21} \nonumber
&\frac{\displaystyle 1}{\displaystyle 2} \partial_t \Vert u_{1,n}^{\epsilon}
\Vert^2_{L^2(\Pi_{\epsilon})}+
\frac{\displaystyle 1}{\displaystyle 2} \partial_t \Vert v_{1,n}^{\epsilon}
\Vert^2_{L^2(\Omega_{\epsilon})}
+\epsilon^2 \, D_1 \Vert
\nabla_x u_{1,n}^{\epsilon} \Vert^2_{L^2(\Pi_{\epsilon})} 
+\tilde{D}_1 \Vert \partial_z u_{1,n}^{\epsilon} \Vert^2_{L^2(\Pi_{\epsilon})}
\\ \nonumber
&+d_1 \Vert \nabla v_{1,n}^{\epsilon} 
\Vert^2_{L^2(\Omega_{\epsilon})} 
\leq C_f+
\epsilon \displaystyle \int_{\Gamma_{\epsilon}} c_1(x,z) (u_{1,n}^{\epsilon}
-v_{1,n}^{\epsilon})_{+} \, (v_{1,n}^{\epsilon}-u_{1,n}^{\epsilon}) 
\, d\sigma_{\epsilon}\\
&+\frac{\displaystyle 1}{\displaystyle 2} \Vert u_{1,n}^{\epsilon}
\Vert^2_{L^2(\Pi_{\epsilon})}
\end{eqnarray}
Let us estimate the second term on the right-hand side of (\ref{4.21}):

\begin{eqnarray} \label{estim} \nonumber
&I=\epsilon \displaystyle \int_{\Gamma_{\epsilon}} c_1(x,z) (u_{1,n}^{\epsilon}
-v_{1,n}^{\epsilon})_{+} \, (v_{1,n}^{\epsilon}-u_{1,n}^{\epsilon})
\, d\sigma_{\epsilon}
\leq \epsilon \displaystyle \int_{\Gamma_{\epsilon}} c_1(x,z)
\vert u_{1,n}^{\epsilon}-v_{1,n}^{\epsilon} \vert^2 \, d\sigma_{\epsilon}
\\ \nonumber
&\leq  \epsilon \Vert c_1(x,z) \Vert_{L^{\infty}(\Gamma_{\epsilon})}
\displaystyle \int_{\Gamma_{\epsilon}} (\vert u_{1,n}^{\epsilon} \vert^2
+\vert v_{1,n}^{\epsilon} \vert^2) \, d\sigma_{\epsilon}\\
&\leq \epsilon \Vert c_1(x,z) \Vert_{L^{\infty}(\Gamma_{\epsilon})}
\bigg( \Vert u_{1,n}^{\epsilon} \Vert^2_{L^2(\Gamma_{\epsilon})}+
\Vert v_{1,n}^{\epsilon} \Vert^2_{L^2(\Gamma_{\epsilon})} \bigg)
\end{eqnarray}
Applying the generalized interpolation-trace inequality (\ref{A.7})
in Appendix A to each
term inside the round brackets, one has:

\begin{equation} \label{4.22}
I \leq C_1 \, \epsilon \, \eta \bigg( \Vert \nabla u_{1,n}^{\epsilon}
\Vert^2_{L^2(\Pi_{\epsilon})}+\Vert \nabla v_{1,n}^{\epsilon}
\Vert^2_{L^2(\Omega_{\epsilon})} \bigg)+C_2 \, \epsilon \, {\eta}^{-1}
\bigg( \Vert  u_{1,n}^{\epsilon}\Vert^2_{L^2(\Pi_{\epsilon})}+
\Vert v_{1,n}^{\epsilon} \Vert^2_{L^2(\Omega_{\epsilon})} \bigg)
\end{equation}
where $\eta$ is a small positive constant.
If we take into account the estimate (\ref{4.22}), the inequality (\ref{4.21})
reads:

\begin{eqnarray} \label{4.23} \nonumber
&\frac{\displaystyle 1}{\displaystyle 2} \partial_t \Vert u_{1,n}^{\epsilon}
\Vert^2_{L^2(\Pi_{\epsilon})}+
\frac{\displaystyle 1}{\displaystyle 2} \partial_t \Vert v_{1,n}^{\epsilon}
\Vert^2_{L^2(\Omega_{\epsilon})}
+(\epsilon^2 \, D_1-C_1 \, \epsilon \, \eta) \Vert
\nabla_x u_{1,n}^{\epsilon} \Vert^2_{L^2(\Pi_{\epsilon})} \\ \nonumber
&+(\tilde{D}_1-C_1 \, \epsilon \, \eta) \Vert
\partial_z u_{1,n}^{\epsilon} \Vert^2_{L^2(\Pi_{\epsilon})}
+(d_1-C_1 \, \epsilon \, \eta) \Vert \nabla v_{1,n}^{\epsilon} 
\Vert^2_{L^2(\Omega_{\epsilon})} 
\leq C_f \\
&+\bigg(C_2 \, \epsilon \, {\eta}^{-1}+
\frac{\displaystyle 1}{\displaystyle 2} \bigg) \Vert  u_{1,n}^{\epsilon}
\Vert^2_{L^2(\Pi_{\epsilon})} 
+C_2 \, \epsilon \, {\eta}^{-1} \Vert v_{1,n}^{\epsilon} 
\Vert^2_{L^2(\Omega_{\epsilon})} 
\end{eqnarray}
If one chooses $\eta < \min \{\frac{\epsilon D_1}{C_1},
\frac{\tilde{D}_1}{\epsilon \, C_1},
\frac{d_1}{\epsilon \, C_1} \}$, the last three terms on the
left-hand side are positive and Eq. (\ref{4.23}) reduces to:

\begin{equation} \label{4.24}
\partial_t \Vert u_{1,n}^{\epsilon}
\Vert^2_{L^2(\Pi_{\epsilon})}+
\partial_t \Vert v_{1,n}^{\epsilon}
\Vert^2_{L^2(\Omega_{\epsilon})} \leq 2 \, C_f+
2 \,C_2 \, \epsilon \, {\eta}^{-1} \bigg( \Vert  u_{1,n}^{\epsilon}
\Vert^2_{L^2(\Pi_{\epsilon})}+\Vert v_{1,n}^{\epsilon}
\Vert^2_{L^2(\Omega_{\epsilon})} \bigg)
\end{equation}
Integrating Eq. (\ref{4.24}) over $[0,t]$ with $t \in [0,T]$,
we get

\begin{eqnarray} \label{4.25} \nonumber
&\Vert u_{1,n}^{\epsilon}(t) \Vert^2_{L^2(\Pi_{\epsilon})}+
\Vert v_{1,n}^{\epsilon}(t)  \Vert^2_{L^2(\Omega_{\epsilon})} \leq
\Vert u_{1,n}^{\epsilon,0} \Vert^2_{L^2(\Pi_{\epsilon})}+
\Vert v_{1,n}^{\epsilon,0}  \Vert^2_{L^2(\Omega_{\epsilon})} \\
&+2 \, C_f \, T
+2 \, C_2 \, \epsilon \, {\eta}^{-1}
\displaystyle \int_0^t ds \bigg(  
\Vert u_{1,n}^{\epsilon}\Vert^2_{L^2(\Pi_{\epsilon})}+
\Vert v_{1,n}^{\epsilon} \Vert^2_{L^2(\Omega_{\epsilon})} \bigg)
\end{eqnarray}
Since the sequences $(u_{1,n}^{\epsilon,0})_{n \in \mathbb{N}}$ and
$(v_{1,n}^{\epsilon,0})_{n \in \mathbb{N}}$ converge in $H^1 (\Pi_{\epsilon})$
and $H^1 (\Omega_{\epsilon})$, respectively, they are bounded in $L^2$.
Therefore, Eq. (\ref{4.25}) reads

\begin{eqnarray} \label{4.25a} \nonumber
&\Vert u_{1,n}^{\epsilon}(t) \Vert^2_{L^2(\Pi_{\epsilon})}+
\Vert v_{1,n}^{\epsilon}(t)  \Vert^2_{L^2(\Omega_{\epsilon})} \leq
C \\
&+2 \, C_2 \, \epsilon \, {\eta}^{-1}
\displaystyle \int_0^t ds \bigg(  
\Vert u_{1,n}^{\epsilon}\Vert^2_{L^2(\Pi_{\epsilon})}+
\Vert v_{1,n}^{\epsilon} \Vert^2_{L^2(\Omega_{\epsilon})} \bigg)
\end{eqnarray}
Applying Gronwalls's inequality, we obtain

\begin{equation} \label{4.25b}
\Vert u_{1,n}^{\epsilon}(t) \Vert^2_{L^2(\Pi_{\epsilon})}+
\Vert v_{1,n}^{\epsilon}(t)  \Vert^2_{L^2(\Omega_{\epsilon})}
\leq C+2 \, C \, C_2 \, \epsilon \, {\eta}^{-1}
\displaystyle \int_0^t e^{2 \, C_2 \, \epsilon \, {\eta}^{-1} (t-s)} \, ds
\end{equation}
Therefore, given $\epsilon \in [0,1]$, $\eta$ is fixed
and for $t \in [0,T]$ we get:

\begin{equation} \label{4.26}  
\Vert u_{1,n}^{\epsilon}(t, \cdot) \Vert^2_{L^2(\Pi_{\epsilon})}+
\Vert v_{1,n}^{\epsilon}(t, \cdot)  \Vert^2_{L^2(\Omega_{\epsilon})} \leq C_3
\end{equation}
where $C_3$ is a positive constant independent of $n$ and $\epsilon$.

By testing Eqs. (\ref{2.7a}) and (\ref{2.8a}),
in the case $1 < i \leq M$,
with $\psi_i=u_{i,n}^{\epsilon}(t, \cdot)$
and $\phi_i=v_{i,n}^{\epsilon}(t, \cdot)$, respectively, one gets:

\begin{eqnarray} \label{add1} \nonumber
&\frac{\displaystyle 1}{\displaystyle 2} \partial_t \Vert u_{i,n}^{\epsilon}
\Vert^2_{L^2(\Pi_{\epsilon})}+\epsilon^2 \, D_i \Vert
\nabla_x u_{i,n}^{\epsilon} \Vert^2_{L^2(\Pi_{\epsilon})} 
+\tilde{D}_i \Vert \partial_z u_{i,n}^{\epsilon} \Vert^2_{L^2(\Pi_{\epsilon})}
\\
&\leq
-\epsilon \displaystyle \int_{\Gamma_{\epsilon}} c_i(x,z) (u_{i,n}^{\epsilon}
-v_{i,n}^{\epsilon})_{+} \, u_{i,n}^{\epsilon} \, d\sigma_{\epsilon}
+C+\frac{\displaystyle 1}{\displaystyle 4}
\Vert u_{i,n}^{\epsilon} \Vert^2_{L^2(\Pi_{\epsilon})}
\end{eqnarray}

\begin{eqnarray} \label{add2} \nonumber
&\frac{\displaystyle 1}{\displaystyle 2} \partial_t \Vert v_{i,n}^{\epsilon}
\Vert^2_{L^2(\Omega_{\epsilon})}+d_i \Vert
\nabla v_{i,n}^{\epsilon} \Vert^2_{L^2(\Omega_{\epsilon})} \leq
\epsilon \displaystyle \int_{\Gamma_{\epsilon}} c_i(x,z) (u_{i,n}^{\epsilon}
-v_{i,n}^{\epsilon})_{+} \, v_{i,n}^{\epsilon} \, d\sigma_{\epsilon}\\
&+\tilde{C}+\frac{\displaystyle 1}{\displaystyle 4}
\Vert v_{i,n}^{\epsilon} \Vert^2_{L^2(\Omega_{\epsilon})}
\end{eqnarray}
where $C$ and $\tilde{C}$ are two positive constants.
Adding the inequalities (\ref{add1}) and (\ref{add2}), and exploiting
the estimate (\ref{estim}) (which holds also when $1 < i \leq M$)
along with the interpolation-trace inequality (\ref{A.7}), we obtain that
also the functions $u_{i,n}^{\epsilon}$ and $v_{i,n}^{\epsilon}$
($1 < i \leq M$) satisfy Eq. (\ref{4.24}).
The rest of the proof carries over verbatim, leading to Eq. (\ref{4.26})
also for the case $1 < i \leq M$.

Thus, we can conclude that:

$\{u_{i,n}^{\epsilon}\}$ is bounded in
$L^{\infty} ([0,T]; L^2(\Pi_{\epsilon}))$, and $\{v_{i,n}^{\epsilon}\}$
is bounded in
$L^{\infty} ([0,T]; L^2(\Omega_{\epsilon}))$.

Let us now derive uniform estimates in $n$ for $\partial_t u_{i,n}^{\epsilon}$,
$\nabla u_{i,n}^{\epsilon}$
and $\partial_t v_{i,n}^{\epsilon}$, $\nabla v_{i,n}^{\epsilon}$
($1 \leq i \leq M$).
In the case $i=1$,
we take in Eq. (\ref{4.5}) $\psi_1=\partial_t u_{1,n}^{\epsilon}(t, \cdot)$ as
test function:

\begin{eqnarray} \label{4.27} \nonumber
&  \displaystyle \int_{\Pi_{\epsilon}} \vert \partial_t u_{1,n}^{\epsilon} 
\vert^2 \, dx \, dz 
+\frac{\displaystyle \epsilon^2}{\displaystyle 2} \displaystyle 
\int_{\Pi_{\epsilon}} D_1 \,
\partial_t (\vert \nabla_x u_{1,n}^{\epsilon} \vert^2) \, dx \, dz \\ \nonumber
&+\frac{\displaystyle 1}{\displaystyle 2} \displaystyle
\int_{\Pi_{\epsilon}} \tilde{D}_1 \,
\partial_t (\vert \partial_z u_{1,n}^{\epsilon} \vert^2) \, dx \, dz
+\epsilon \displaystyle \int_{\Gamma_{\epsilon}} c_1(x,z) 
(u_{1,n}^{\epsilon}-v_{1,n}^{\epsilon})_{+} \, (\partial_t u_{1,n}^{\epsilon})
\, d\sigma_{\epsilon} \\ 
&=-\displaystyle \int_{\Pi_{\epsilon}}
w_{\tilde M}^{\epsilon} \, (\partial_t u_{1,n}^{\epsilon}) \, dx \, dz
+\displaystyle \int_{\Pi_{\epsilon}} f^{\epsilon}(t,x,z) \, (\partial_t 
u_{1,n}^{\epsilon}) \, dx \, dz
\end{eqnarray}
where $w_{\tilde M}^{\epsilon}=:\sum_{j=1}^{M} a_{1,j} \, 
\sigma_{\tilde M} (u_{1,n}^{\epsilon}) \, 
\sigma_{\tilde M} (u_{j,n}^{\epsilon})$.

By taking into account the H\"older and Young inequalities, Eq. (\ref{4.27})
becomes
\begin{eqnarray} \label{4.28} \nonumber
&  \displaystyle \int_{\Pi_{\epsilon}} \vert \partial_t u_{1,n}^{\epsilon} 
\vert^2  \, dx \, dz
+\frac{\displaystyle \epsilon^2 \, D_1}{\displaystyle 2} \displaystyle 
\int_{\Pi_{\epsilon}}  \,
\partial_t (\vert \nabla_x u_{1,n}^{\epsilon} \vert^2) \, dx \, dz
+\frac{\displaystyle \tilde{D}_1}{\displaystyle 2} \displaystyle
\int_{\Pi_{\epsilon}}  \,
\partial_t (\vert \partial_z u_{1,n}^{\epsilon} \vert^2) \, dx \, dz
\\ \nonumber
&\leq
\frac{\displaystyle 1}{\displaystyle 2}
\Vert w_{\tilde M}^{\epsilon} \Vert^2_{L^2(\Pi_{\epsilon})}+
\frac{\displaystyle 1}{\displaystyle 2}
\Vert \partial_t u_{1,n}^{\epsilon} \Vert^2_{L^2(\Pi_{\epsilon})} 
-\epsilon \displaystyle \int_{\Gamma_{\epsilon}} c_1(x,z) 
(u_{1,n}^{\epsilon}-v_{1,n}^{\epsilon})_{+} \, 
(\partial_t u_{1,n}^{\epsilon}) \, d\sigma_{\epsilon}
\\
&+\eta^{-1} \Vert f^{\epsilon}(t,x,z) \Vert^2_{L^2(\Pi_{\epsilon})} +
\eta \Vert \partial_t u_{1,n}^{\epsilon} \Vert^2_{L^2(\Pi_{\epsilon})}
\end{eqnarray}
Exploiting the truncation of the coagulation terms and choosing
$\eta=\frac{\displaystyle 1}{\displaystyle 4}$, we get:

\begin{eqnarray} \label{4.29} \nonumber
& \frac{\displaystyle 1}{\displaystyle 4}
\Vert \partial_t u_{1,n}^{\epsilon} \Vert^2_{L^2(\Pi_{\epsilon})} 
+\frac{\displaystyle \epsilon^2}{\displaystyle 2} \,  D_1 \,
\partial_t \displaystyle 
\int_{\Pi_{\epsilon}} 
\vert \nabla_x u_{1,n}^{\epsilon} \vert^2 \, dx \, dz
+\frac{\displaystyle \tilde{D}_1}{\displaystyle 2} \, 
\partial_t \displaystyle \int_{\Pi_{\epsilon}} \vert \partial_z 
u_{1,n}^{\epsilon} \vert^2 \, dx \, dz\\
&\leq C^1_{\tilde M} 
-\epsilon \displaystyle \int_{\Gamma_{\epsilon}} c_1(x,z) 
(u_{1,n}^{\epsilon}-v_{1,n}^{\epsilon})_{+} \, 
(\partial_t u_{1,n}^{\epsilon}) \, d\sigma_{\epsilon}
+4 \, \Vert f^{\epsilon}(t,x,z) \Vert^2_{L^2(\Pi_{\epsilon})}
\end{eqnarray}
where $C^1_{\tilde M}$ is a positive constant which depends on $\tilde M$.

Let us now test Eq. (\ref{4.10}) with the function
$\phi_1=\partial_t v_{1,n}^{\epsilon}(t, \cdot)$:

\begin{eqnarray} \label{4.30} \nonumber
&  \displaystyle \int_{\Omega_{\epsilon}} \vert \partial_t v_{1,n}^{\epsilon} 
\vert^2  \, dx \, dz
+\frac{\displaystyle d_1}{\displaystyle 2} \displaystyle 
\int_{\Omega_{\epsilon}} 
\partial_t (\vert \nabla v_{1,n}^{\epsilon} \vert^2) \, dx \, dz \\ \nonumber
&-\epsilon \displaystyle \int_{\Gamma_{\epsilon}} c_1(x,z) 
(u_{1,n}^{\epsilon}-v_{1,n}^{\epsilon})_{+} \, 
(\partial_t v_{1,n}^{\epsilon}) \, d\sigma_{\epsilon}
\\  
&=-\displaystyle \int_{\Omega_{\epsilon}}
\bigg[ \sum_{j=1}^{M} b_{1,j} \,
\sigma_{\tilde M} (v_{1,n}^{\epsilon}) \,
\sigma_{\tilde M} (v_{j,n}^{\epsilon}) \bigg]
\, (\partial_t v_{1,n}^{\epsilon}) \, dx \, dz
\end{eqnarray}
By applying once again the H\"older and Young inequalities to the
right-hand side as above, and exploiting Eq. (\ref{2.6}), we end up with the
following expression:

\begin{eqnarray} \label{4.31}  \nonumber
& \frac{\displaystyle 1}{\displaystyle 2}
\Vert \partial_t v_{1,n}^{\epsilon} \Vert^2_{L^2(\Omega_{\epsilon})} 
+\frac{\displaystyle d_1}{\displaystyle 2} \, 
\partial_t \displaystyle 
\int_{\Omega_{\epsilon}} 
\vert \nabla v_{1,n}^{\epsilon} \vert^2 \, dx \, dz \\
&\leq C^2_{\tilde M} 
+\epsilon \displaystyle \int_{\Gamma_{\epsilon}} c_1(x,z) 
(u_{1,n}^{\epsilon}-v_{1,n}^{\epsilon})_{+} \, (\partial_t v_{1,n}^{\epsilon})
\, d\sigma_{\epsilon}
\end{eqnarray}
where $C^2_{\tilde M}$ is a positive constant which depends on $\tilde M$.
Adding (\ref{4.29}) and (\ref{4.31}), and taking into account the
$L^2$-boundedness of $f^{\epsilon}(t,x,z)$,
one obtains:

\begin{eqnarray} \label{4.32} \nonumber
&\frac{\displaystyle 1}{\displaystyle 4} \Vert  \partial_t u_{1,n}^{\epsilon}
\Vert^2_{L^2(\Pi_{\epsilon})}+
\frac{\displaystyle 1}{\displaystyle 2} \Vert \partial_t v_{1,n}^{\epsilon}
\Vert^2_{L^2(\Omega_{\epsilon})}
+\frac{\displaystyle \epsilon^2}{\displaystyle 2} \,  D_1 \,
\partial_t \displaystyle \int_{\Pi_{\epsilon}} 
\vert \nabla_x u_{1,n}^{\epsilon} \vert^2 \, dx \, dz\\ \nonumber
&+\frac{\displaystyle \tilde{D}_1}{\displaystyle 2}
\partial_t \displaystyle \int_{\Pi_{\epsilon}}
\vert \partial_z u_{1,n}^{\epsilon} \vert^2 \, dx \, dz
+\frac{\displaystyle d_1}{\displaystyle 2} \, 
\partial_t \displaystyle \int_{\Omega_{\epsilon}} 
\vert \nabla v_{1,n}^{\epsilon} \vert^2  \, dx \, dz\\
&\leq C+
\epsilon \displaystyle \int_{\Gamma_{\epsilon}} c_1(x,z) (u_{1,n}^{\epsilon}
-v_{1,n}^{\epsilon})_{+} \, (\partial_t v_{1,n}^{\epsilon}-
\partial_t u_{1,n}^{\epsilon}) \, d\sigma_{\epsilon}
\end{eqnarray}
where $C$ is a positive constant.
If we decompose now the function $\partial_t (v_{1,n}^{\epsilon}-
u_{1,n}^{\epsilon})$ on the right-hand side in its positive and negative
parts, Eq. (\ref{4.32}) can be rewritten as:

\begin{eqnarray} \label{4.33} \nonumber
&\frac{\displaystyle 1}{\displaystyle 4} \Vert  \partial_t u_{1,n}^{\epsilon}
\Vert^2_{L^2(\Pi_{\epsilon})}+
\frac{\displaystyle 1}{\displaystyle 2} \Vert \partial_t v_{1,n}^{\epsilon}
\Vert^2_{L^2(\Omega_{\epsilon})}
+\frac{\displaystyle \epsilon^2}{\displaystyle 2} \,  D_1 \,
\partial_t \displaystyle \int_{\Pi_{\epsilon}} 
\vert \nabla_x u_{1,n}^{\epsilon} \vert^2 \, dx \, dz \\
&+\frac{\displaystyle \tilde{D}_1}{\displaystyle 2} \, 
\partial_t \displaystyle \int_{\Pi_{\epsilon}} 
\vert \partial_z u_{1,n}^{\epsilon} \vert^2 \, dx \, dz
+\frac{\displaystyle d_1}{\displaystyle 2} \, 
\partial_t \displaystyle \int_{\Omega_{\epsilon}} 
\vert \nabla v_{1,n}^{\epsilon} \vert^2 \, dx \, dz \\
&\leq C
-\frac{\displaystyle \epsilon}{\displaystyle 2} 
\displaystyle \int_{\Gamma_{\epsilon}} c_1(x,z) 
\partial_t [ (u_{1,n}^{\epsilon}-v_{1,n}^{\epsilon})^2_{+} ] \, 
d\sigma_{\epsilon}
\end{eqnarray}
Integrating over $[0,t]$ with $t \in [0,T]$, we deduce:

\begin{eqnarray} \label{4.34} \nonumber
&\frac{\displaystyle 1}{\displaystyle 4} \displaystyle \int_{0}^{t} ds \,
\Vert  \partial_s u_{1,n}^{\epsilon}
\Vert^2_{L^2(\Pi_{\epsilon})}+
\frac{\displaystyle 1}{\displaystyle 2} \displaystyle \int_{0}^{t} ds \,
\Vert \partial_s v_{1,n}^{\epsilon}
\Vert^2_{L^2(\Omega_{\epsilon})} \\ \nonumber
&+\frac{\displaystyle \epsilon^2}{\displaystyle 2} \,  D_1 \,
\displaystyle \int_{\Pi_{\epsilon}} 
\vert \nabla_x u_{1,n}^{\epsilon} \vert^2 \, dx \, dz
-\frac{\displaystyle \epsilon^2}{\displaystyle 2} \,  D_1 \,
\displaystyle \int_{\Pi_{\epsilon}}
\vert \nabla_x u_{1,n}^{\epsilon}(0) \vert^2 \, dx \, dz \\ \nonumber
&+\frac{\displaystyle \tilde{D}_1}{\displaystyle 2} \,
\displaystyle \int_{\Pi_{\epsilon}}
\vert \partial_z u_{1,n}^{\epsilon} \vert^2 \, dx \, dz-
\frac{\displaystyle \tilde{D}_1}{\displaystyle 2} \,
\displaystyle \int_{\Pi_{\epsilon}}
\vert \partial_z u_{1,n}^{\epsilon}(0) \vert^2 \, dx \, dz\\ \nonumber
&+\frac{\displaystyle d_1}{\displaystyle 2} \, 
\displaystyle \int_{\Omega_{\epsilon}} 
\vert \nabla v_{1,n}^{\epsilon} \vert^2 \, dx \, dz-
\frac{\displaystyle d_1}{\displaystyle 2} \,
\displaystyle \int_{\Omega_{\epsilon}}
\vert \nabla v_{1,n}^{\epsilon}(0) \vert^2 \, dx \, dz \leq
C \, T \\
&+\frac{\displaystyle \epsilon}{\displaystyle 2} 
\displaystyle \int_{\Gamma_{\epsilon}} c_1(x,z) 
(u_{1,n}^{\epsilon}-v_{1,n}^{\epsilon})^2_{+} (0) \, d\sigma_{\epsilon}
-\frac{\displaystyle \epsilon}{\displaystyle 2}
\displaystyle \int_{\Gamma_{\epsilon}} c_1(x,z)
(u_{1,n}^{\epsilon}-v_{1,n}^{\epsilon})^2_{+} \, d\sigma_{\epsilon}
\end{eqnarray}
Hence, taking into account that the last term on the right-hand side
of Eq.(\ref{4.34}) is negative, one has

\begin{eqnarray} \label{4.34a} \nonumber
& \displaystyle \int_{0}^{t} ds \,
\Vert  \partial_s u_{1,n}^{\epsilon}
\Vert^2_{L^2(\Pi_{\epsilon})}+
2 \, \displaystyle \int_{0}^{t} ds \,
\Vert \partial_s v_{1,n}^{\epsilon}
\Vert^2_{L^2(\Omega_{\epsilon})}
+2 \,{\displaystyle \epsilon^2} \,  D_1 \,
\Vert \nabla_x u_{1,n}^{\epsilon} \Vert^2_{L^2(\Pi_{\epsilon})} \\ \nonumber
&+2 \, \tilde{D}_1 \Vert \partial_z u_{1,n}^{\epsilon} 
\Vert^2_{L^2(\Pi_{\epsilon})}
+2 \,d_1 \Vert \nabla v_{1,n}^{\epsilon} 
\Vert^2_{L^2(\Omega_{\epsilon})} 
\leq 4 \, C \, T \\ \nonumber
&+2 \, D_1 \,\displaystyle \int_{\Pi_{\epsilon}}
\vert \nabla_x u_{1,n}^{\epsilon}(0) \vert^2 \, dx \, dz+
2 \, \tilde{D}_1 \, \displaystyle \int_{\Pi_{\epsilon}}
\vert \partial_z u_{1,n}^{\epsilon}(0) \vert^2 \, dx \, dz \\
&+2 \, d_1 \,\displaystyle \int_{\Omega_{\epsilon}}
\vert \nabla v_{1,n}^{\epsilon}(0) \vert^2 \, dx \, dz
+2 \,\displaystyle \int_{\Gamma_{\epsilon}} c_1(x,z)
(u_{1,n}^{\epsilon}-v_{1,n}^{\epsilon})^2_{+}(0) \, d\sigma_{\epsilon}
\end{eqnarray}
Since the sequences
$u_{1,n}^{\epsilon}(0)$ and $v_{1,n}^{\epsilon}(0)$ are bounded, it follows:

\begin{eqnarray} \label{4.35} \nonumber
&\displaystyle \int_{0}^{T} ds \, \Vert  \partial_s u_{1,n}^{\epsilon}
\Vert^2_{L^2(\Pi_{\epsilon})}+
2 \,\displaystyle \int_{0}^{T} ds \, \Vert \partial_s v_{1,n}^{\epsilon}
\Vert^2_{L^2(\Omega_{\epsilon})}+
2 \, \epsilon^2 \,  D_1 \, 
\Vert \nabla_x u_{1,n}^{\epsilon} \Vert^2_{L^2(\Pi_{\epsilon})} \\
&+2 \, \tilde{D}_1 \, \Vert \partial_z u_{1,n}^{\epsilon} 
\Vert^2_{L^2(\Pi_{\epsilon})}
+2 \, d_1 \,
\Vert \nabla v_{1,n}^{\epsilon} \Vert^2_{L^2(\Omega_{\epsilon})} 
\leq \tilde {C} \; \; \; \; \; \; \mbox{for} \; \;t \in [0,T]
\end{eqnarray}
where $\tilde {C}$ is a  positive constant independent of $n$ and $\epsilon$.

In the case $1 < i \leq M$, by testing Eqs. (\ref{2.7a}) and (\ref{2.8a})
with $\psi_i=\partial_t u_{i,n}^{\epsilon}(t, \cdot)$ and
$\phi_i=\partial_t v_{i,n}^{\epsilon}(t, \cdot)$, respectively, one obtains:

\begin{eqnarray} \label{add3} \nonumber
& \frac{\displaystyle 1}{\displaystyle 4}
\Vert \partial_t u_{i,n}^{\epsilon} \Vert^2_{L^2(\Pi_{\epsilon})} 
+\frac{\displaystyle \epsilon^2}{\displaystyle 2} \,  D_i \,
\partial_t \displaystyle 
\int_{\Pi_{\epsilon}} 
\vert \nabla_x u_{i,n}^{\epsilon} \vert^2 \, dx \, dz
+\frac{\displaystyle \tilde{D}_i}{\displaystyle 2} \, 
\partial_t \displaystyle \int_{\Pi_{\epsilon}} \vert \partial_z 
u_{i,n}^{\epsilon} \vert^2 \, dx \, dz\\
&\leq C 
-\epsilon \displaystyle \int_{\Gamma_{\epsilon}} c_i(x,z) 
(u_{i,n}^{\epsilon}-v_{i,n}^{\epsilon})_{+} \, (\partial_t u_{i,n}^{\epsilon})
\, d\sigma_{\epsilon}
\end{eqnarray}

\begin{eqnarray} \label{add4} \nonumber
& \frac{\displaystyle 1}{\displaystyle 4}
\Vert \partial_t v_{i,n}^{\epsilon} \Vert^2_{L^2(\Omega_{\epsilon})} 
+\frac{\displaystyle d_i}{\displaystyle 2} \, 
\partial_t \displaystyle 
\int_{\Omega_{\epsilon}} 
\vert \nabla v_{i,n}^{\epsilon} \vert^2 \, dx \, dz \\
&\leq \tilde{C} 
+\epsilon \displaystyle \int_{\Gamma_{\epsilon}} c_i(x,z) 
(u_{i,n}^{\epsilon}-v_{i,n}^{\epsilon})_{+} \, (\partial_t v_{i,n}^{\epsilon})
\, d\sigma_{\epsilon}
\end{eqnarray}
due to the boundedness of the coagulation terms given by (\ref{2.6}).
In Eqs. (\ref{add3}) and (\ref{add4}), $C$ and $\tilde{C}$ are two positive
constants which depend on $\tilde{M}$.
Adding the two inequalities (\ref{add3}) and (\ref{add4}), and applying
exactly the same arguments considered for $i=1$, we obtain that
also the functions $u_{i,n}^{\epsilon}$ and $v_{i,n}^{\epsilon}$
($1 < i \leq M$) satisfy Eq. (\ref{4.35}).

Thus, combining the estimates (\ref{4.26}) and (\ref{4.35}), one concludes
that:

$\{u_{i,n}^{\epsilon}\}$ is bounded in
$H^1 ([0,T]; L^2(\Pi_{\epsilon})) \cap L^{\infty} ([0,T]; 
H^1(\Pi_{\epsilon}))$,
and $\{v_{i,n}^{\epsilon}\}$ is bounded in
$H^1 ([0,T]; L^2(\Omega_{\epsilon})) \cap L^{\infty} ([0,T];
H^1(\Omega_{\epsilon}))$.
Hence, we can state the following proposition.

\begin{proposition} \label{conv}
Since $(u_{i,n}^{\epsilon})_{n\in \mathbb N}$ is bounded in
$L^\infty([0,T]; H^1(\Pi_{\epsilon}))$, by the Banach-Alaoglu theorem
we may assume that, up to a subsequence,
$u_{i,n}^{\epsilon} \rightharpoonup u_i^{\epsilon}$ weakly*
in
$L^\infty([0,T]; H^1(\Pi_{\epsilon})) $, i.e.
for all $\psi\in L^1([0,T]; H^1(\Pi_{\epsilon}))$

\begin{equation}\label{banach} \begin{split}
\int_0^T \int_{\Pi_{\epsilon}} & \big( \psi(t,x,z) u_{i,n}^{\epsilon} (t,x,z)+ 
\nabla \psi(t,x,z) \cdot \nabla u_{i,n}^{\epsilon} (t,x,z)\big) \, dt \, dx\, 
dz 
\\&
\longrightarrow \int_0^T \int_{\Pi_{\epsilon}} \big( \psi(t,x,z) u_i^{\epsilon} 
(t,x,z)+ \nabla \psi(t,x,z)\cdot\nabla u_i^{\epsilon}(t,x,z)\big) \, dt \, 
dx\, dz 
\end{split}
\end{equation}
as $n\to\infty$.
Since $ L^2([0,T]; H^1(\Pi_{\epsilon}))\subset  
L^1([0,T]; H^1(\Pi_{\epsilon}))$, formula \eqref{banach} holds
also for $\psi \in  L^2([0,T]; H^1(\Pi_{\epsilon}))$.

In addition, $(u_{i,n}^{\epsilon})_{n\in \mathbb N}$ is bounded in
$H^1([0,T]; L^2(\Pi_{\epsilon}))$, thus we can also assume
that $u_{i,n}^{\epsilon} \rightharpoonup u_i^{\epsilon}$ weakly in
$ H^1([0,T]; L^2(\Pi_{\epsilon}))$ as $n\to\infty$.

Finally, since in particular
$$
\mbox{$(u_{i,n}^{\epsilon})_{n\in \mathbb N}$ is bounded in}
\quad H^1([0,T]; L^2(\Pi_{\epsilon})) \cap
L^\infty([0,T]; H^1(\Pi_{\epsilon})),
$$
by the Aubin-Lions-Simon theorem (\cite{BF}, Theorem II.5.16, p. 102)
(see Appendix B)  we can infer that

$u_{i,n}^{\epsilon} \rightarrow u_i^{\epsilon}$ strongly in
$C^0([0,T]; L^2(\Pi_{\epsilon}))$ as $n\to\infty$.

\end{proposition}
An analogous proposition can be proved also for the sequence
$(v_{i,n}^{\epsilon})_{n\in \mathbb N}$.

Now, integrating Eq. (\ref{4.5}) with respect to time and using as a test
function $\psi_1=\phi(t) \zeta_1 (x,z)$, with $\phi \in {\mathcal D} ([0,T])$
and $\zeta_1 \in H^1(\Pi_{\epsilon})$, we get

\begin{eqnarray} \label{4.36} \nonumber
&\displaystyle \int_{0}^{T} \int_{\Pi_{\epsilon}} 
\partial_t u_{1,n}^{\epsilon} \, \phi(t) \, \zeta_1(x,z) \, dt \, dx \, dz+
\epsilon^2 \, \displaystyle \int_{0}^{T} \int_{\Pi_{\epsilon}} D_1 
\,\nabla_x u_{1,n}^{\epsilon} \cdot \nabla_x \zeta_1 \, \phi(t) 
\, dt \, dx \, dz \\ \nonumber
&+\displaystyle \int_{0}^{T} \int_{\Pi_{\epsilon}} \tilde{D}_1 \,
\partial_z u_{1,n}^{\epsilon} \cdot \partial_z \zeta_1 \, \phi(t)
\, dt \, dx \, dz \\ \nonumber
&+ \epsilon \,
\displaystyle \int_{0}^{T} \int_{\Gamma_{\epsilon}} c_1(x,z) \,
(u_{1,n}^{\epsilon}-v_{1,n}^{\epsilon})_{+} \, \phi(t) \, \zeta_1(x,z) 
\, dt \, d\sigma_{\epsilon}\\ \nonumber
&=\displaystyle \int_{0}^{T} \int_{\Pi_{\epsilon}} 
L_1^{\tilde M} (u_{i,n}^{\epsilon}) \, \phi(t) \, \zeta_1(x,z) 
\, dt \, dx \, dz \\
&+\displaystyle \int_{0}^{T} \int_{\Pi_{\epsilon}}
f^{\epsilon}(t,x,z) \, \phi(t) \, \zeta_1(x,z) \, dt \, dx \, dz
\end{eqnarray}
where we denote by ${\mathcal D} ([0,T])$ the set of indefinitely
differentiable functions whose support is a compact set.

Exploiting the convergence results stated in Proposition \ref{conv}
we can pass to the limit
as $n \rightarrow \infty$ to obtain:

\begin{eqnarray} \label{4.37} \nonumber
&\displaystyle \int_{0}^{T} \int_{\Pi_{\epsilon}} 
\partial_t u_{1}^{\epsilon} \, \phi(t) \, \zeta_1(x,z)
\, dt \, dx \, dz+
\epsilon^2 \, \displaystyle \int_{0}^{T} \int_{\Pi_{\epsilon}} D_1 
\,\nabla_x u_{1}^{\epsilon} \cdot \nabla_x \zeta_1 \, \phi(t) 
\, dt \, dx \, dz\\ \nonumber
&+\displaystyle \int_{0}^{T} \int_{\Pi_{\epsilon}} \tilde{D}_1 \,
\partial_z u_{1}^{\epsilon} \cdot \partial_z \zeta_1 \, \phi(t)
\, dt \, dx \, dz \\ \nonumber
&+ \epsilon \,
\displaystyle \int_{0}^{T} \int_{\Gamma_{\epsilon}} c_1(x,z) \,
(u_{1}^{\epsilon}-v_{1}^{\epsilon})_{+} \, \phi(t) \, \zeta_1(x,z) 
\,dt \, d\sigma_{\epsilon} \\ \nonumber
&=\displaystyle \int_{0}^{T} \int_{\Pi_{\epsilon}} 
L_1^{\tilde M} (u_{i}^{\epsilon}) \, \phi(t) \, \zeta_1(x,z) 
\, dt \, dx \, dz \\ 
&+\displaystyle \int_{0}^{T} \int_{\Pi_{\epsilon}} f^{\epsilon}(t,x,z) \, 
\phi(t) \, \zeta_1(x,z) \, dt \, dx \, dz
\end{eqnarray}
where we have taken into account that the term
$L_1^{\tilde M} (u_{i}^{\epsilon})$ is Lipschitz continuous.
Eq. (\ref{4.37}), which holds for arbitrary $\phi(t) \in {\mathcal D}([0,T])$
and $\zeta_1(x,z) \in H^1(\Pi_{\epsilon})$, is exactly the variational
equation (\ref{2.7a}).

By using the same arguments handled above for
$u_{1}^{\epsilon}$, we can prove that the functions $v_{1}^{\epsilon}$
and $u_{i}^{\epsilon}$,
$v_{i}^{\epsilon}$ ($1 < i \leq M$) satisfy Eqs. (\ref{2.7a}), (\ref{2.8a}).

It remains to show that the initial conditions hold.
By the Aubin-Lions-Simon theorem it follows that
$u_{1,n}^{\epsilon} \rightarrow u_1^{\epsilon}$ strongly in
$C^0([0,T]; L^2(\Pi_{\epsilon}))$ as $n\to\infty$.
Since $u_{1,n}^{\epsilon}(t=0)=u_{1,n}^{\epsilon,0}$ and
$u_{1,n}^{\epsilon,0} \rightarrow U^{\epsilon}_1(x,z)$ strongly in
$H^1(\Pi_{\epsilon})$
as $n\to\infty$, we conclude that $u_{1}^{\epsilon}(0,x,z)=U^{\epsilon}_1(x,z)$.
Following the same line of arguments, we also obtain
$v_{1}^{\epsilon}(0,x,z)=0$ along with
$u_{i}^{\epsilon}(0,x,z)=v_{i}^{\epsilon}(0,x,z)=0$ ($1 < i \leq M$).

\end{proof}

\section{Positivity and boundedness of solutions}\label{sec3}

\begin{lemma} \label{l3.1}
For a given small $\epsilon >0$,
let $u_i^{\epsilon}(t,x,z)$ and $v_i^{\epsilon}(t,x,z)$ ($1 \leq i \leq M$)
be  solutions of the system (\ref{2.1a})
in the sense of the Definition \ref{d2.1}.
Then $0 \leq u_i^{\epsilon}(t,x,z) < M_i$ a.e. in $(0,T) \times \Pi_{\epsilon}$
and $0 \leq v_i^{\epsilon}(t,x,z) < \overline{M}_i$ a.e. in $(0,T) \times 
\Omega_{\epsilon}$, where $M_i >0$ and $\overline{M}_i >0$ are
constants independent of $\tilde M$ and $\epsilon$.
\end{lemma}

\begin{proof}
In the case $i=1$,
let us test Eq. (\ref{2.7a}) with the function
$\psi_1=-u_1^{\epsilon, -}(t, \cdot)$:

\begin{eqnarray} \nonumber
&\displaystyle \int_{\Pi_{\epsilon}} \partial_t u_1^{\epsilon} \, 
(-u_1^{\epsilon, -}) \, dx \, dz+
\epsilon^2 \displaystyle \int_{\Pi_{\epsilon}} D_1 
\nabla_x u_1^{\epsilon} \cdot \nabla_x (-u_1^{\epsilon, -}) \, dx \, dz 
\\\nonumber
&+\displaystyle \int_{\Pi_{\epsilon}} \tilde{D}_1 \, 
\partial_z u_1^{\epsilon} \cdot \partial_z (-u_1^{\epsilon, -})
\, dx \, dz
+\epsilon \displaystyle \int_{\Gamma_{\epsilon}} c_1(x,z) 
(u_1^{\epsilon}-v_1^{\epsilon})_{+} \, (-u_1^{\epsilon, -})   
\, d\sigma_{\epsilon} \\
&=\displaystyle \int_{\Pi_{\epsilon}} L_1^{\tilde M}(u^{\epsilon}) \, 
(-u_1^{\epsilon, -}) \, dx \, dz
+\displaystyle \int_{\Pi_{\epsilon}} 
f^{\epsilon}(t,x,z) \, (-u_1^{\epsilon, -}) \, dx \, dz
\end{eqnarray}
Decomposing the function $u_1^{\epsilon}$ in its positive and negative parts,
we get:

\begin{eqnarray} \label{2.9} \nonumber
&\displaystyle \int_{\Pi_{\epsilon}} \frac{\partial u_1^{\epsilon, -}}
{\partial t} \, u_1^{\epsilon, -} \, dx \, dz
+\epsilon^2 \displaystyle \int_{\Pi_{\epsilon}}
 D_1 \, \vert \nabla_x u_1^{\epsilon, -} \vert^2 \, dx \, dz \\\nonumber
&+\displaystyle \int_{\Pi_{\epsilon}} \tilde{D}_1 \, \vert \partial_z
u_1^{\epsilon, -} \vert^2 \, dx \, dz 
= \epsilon \displaystyle \int_{\Gamma_{\epsilon}} c_1(x,z) (u_1^{\epsilon}-
v_1^{\epsilon})_{+}
\, u_1^{\epsilon, -} \, d\sigma_{\epsilon} \\
&+\displaystyle \int_{\Pi_{\epsilon}} \bigg[ \sum_{j=1}^{M} a_{1,j} \,
\sigma_{\tilde M} (u_1^{\epsilon}) \, \sigma_{\tilde M} (u_j^{\epsilon}) 
\bigg] \, u_1^{\epsilon, -} \, dx \, dz
-\displaystyle \int_{\Pi_{\epsilon}}
f^{\epsilon}(t,x,z) \, u_1^{\epsilon, -} \, dx \, dz
\end{eqnarray}
Since the last but one term
on the right-hand side  is always zero and the last one is negative,
one obtains:

\begin{eqnarray} \label{2.10} \nonumber
&\frac{\displaystyle 1}{\displaystyle 2} \partial_t \Vert u_1^{\epsilon, -}
\Vert^2_{L^2(\Pi_{\epsilon})}+\epsilon^2 D_1 \, \Vert
\nabla_x u_1^{\epsilon, -} \Vert^2_{L^2(\Pi_{\epsilon})} 
+\tilde{D}_1 \, \Vert \partial_z u_1^{\epsilon, -} 
\Vert^2_{L^2(\Pi_{\epsilon})} \\
&\leq
\epsilon \displaystyle \int_{\Gamma_{\epsilon}} c_1(x,z) \, (u_1^{\epsilon}-
v_1^{\epsilon})_{+} \, u_1^{\epsilon, -} \, d\sigma_{\epsilon}
\end{eqnarray}

Let us now take $\phi_1=-v_1^{\epsilon, -}(t, \cdot)$ as test function in
Eq. (\ref{2.8a}) and decompose the function $v_1^{\epsilon}$ in its
positive and negative parts:

\begin{eqnarray} \label{2.11} \nonumber
&\displaystyle \int_{\Omega_{\epsilon}} \frac{\partial v_1^{\epsilon, -}}
{\partial t} \, v_1^{\epsilon, -} \, dx \, dz
+ \displaystyle \int_{\Omega_{\epsilon}}
 d_1 \, \vert \nabla v_1^{\epsilon, -} \vert^2 \, dx \, dz 
=
-\epsilon \displaystyle \int_{\Gamma_{\epsilon}} c_1(x,z) 
(u_1^{\epsilon}-v_1^{\epsilon})_{+}
\, v_1^{\epsilon, -} \, d\sigma_{\epsilon} \\
&+\displaystyle \int_{\Omega_{\epsilon}} \bigg[ \sum_{j=1}^{M} b_{1,j} \,
\sigma_{\tilde M} (v_1^{\epsilon}) \, \sigma_{\tilde M} (v_j^{\epsilon}) 
\bigg] \, v_1^{\epsilon, -} \, dx \, dz
\end{eqnarray}
Since the last term on the right-hand side  is always zero, Eq. (\ref{2.11})
reduces to:

\begin{equation} \label{2.12}
\frac{\displaystyle 1}{\displaystyle 2} \partial_t \Vert v_1^{\epsilon, -}
\Vert^2_{L^2(\Omega_{\epsilon})}+d_1 \, \Vert
\nabla v_1^{\epsilon, -} \Vert^2_{L^2(\Omega_{\epsilon})}=
-\epsilon \displaystyle \int_{\Gamma_{\epsilon}} c_1(x,z) \, (u_1^{\epsilon}-
v_1^{\epsilon})_{+} \, v_1^{\epsilon, -} \, d\sigma_{\epsilon}
\end{equation}
Adding Eqs. (\ref{2.10}) and (\ref{2.12}) we obtain:

\begin{eqnarray} \label{2.13} \nonumber
&\frac{\displaystyle 1}{\displaystyle 2} \partial_t \Vert u_1^{\epsilon, -}
\Vert^2_{L^2(\Pi_{\epsilon})}+
\frac{\displaystyle 1}{\displaystyle 2} \partial_t \Vert v_1^{\epsilon, -}
\Vert^2_{L^2(\Omega_{\epsilon})}+
\epsilon^2 D_1 \, \Vert
\nabla_x u_1^{\epsilon, -} \Vert^2_{L^2(\Pi_{\epsilon})} \\ \nonumber
&+\tilde{D}_1 \, \Vert \partial_z u_1^{\epsilon, -} 
\Vert^2_{L^2(\Pi_{\epsilon})}
+d_1 \, \Vert
\nabla v_1^{\epsilon, -} \Vert^2_{L^2(\Omega_{\epsilon})} \\
&\leq \epsilon \displaystyle \int_{\Gamma_{\epsilon}} c_1(x,z) \, 
(u_1^{\epsilon}-
v_1^{\epsilon})_{+} \, (u_1^{\epsilon, -}-v_1^{\epsilon, -}) \, 
d\sigma_{\epsilon}
\end{eqnarray}
Let us now estimate the term on the right-hand side of Eq. (\ref{2.13}):

\begin{eqnarray} \label{2.14} \nonumber
&\epsilon \displaystyle \int_{\Gamma_{\epsilon}} c_1(x,z) \, (u_1^{\epsilon}-
v_1^{\epsilon})_{+} \, (u_1^{\epsilon, -}-v_1^{\epsilon, -}) 
\, d\sigma_{\epsilon} 
\leq \epsilon \displaystyle \int_{\Gamma_{\epsilon}} c_1(x,z) \,
{\mathcal H}(u_1^{\epsilon}-v_1^{\epsilon}) (u_1^{\epsilon}-v_1^{\epsilon}) \,
u_1^{\epsilon, -} \, d\sigma_{\epsilon} \\ 
&\leq \epsilon \displaystyle \int_{\Gamma_{\epsilon}} c_1(x,z) \,
{\mathcal H}(u_1^{\epsilon}-v_1^{\epsilon}) (u_1^{\epsilon, -} \,
v_1^{\epsilon, -}) \, d\sigma_{\epsilon} 
\leq \epsilon \Vert c_1(x,z) \Vert_{L^{\infty}(\Gamma_{\epsilon})} 
\displaystyle \int_{\Gamma_{\epsilon}} u_1^{\epsilon, -} \, v_1^{\epsilon, -}
\, d\sigma_{\epsilon}
\end{eqnarray}
where we denote by ${\mathcal H}(\cdot)$ the Heaviside function:

\begin{equation} \nonumber
{\mathcal H}(s):=\begin{cases}
0, \; \; \; \; s<0 \\
1, \; \; \; \; s \geq 0. 
\end{cases}
\end{equation}
Exploiting the H\"older and  Young
inequalities along with  the generalized interpolation-trace inequality
(\ref{A.7}),
Eq. (\ref{2.14}) becomes

\begin{eqnarray} \label{2.22} \nonumber
&\epsilon \displaystyle \int_{\Gamma_{\epsilon}} c_1(x,z) \, (u_1^{\epsilon}-
v_1^{\epsilon})_{+} \, (u_1^{\epsilon, -}-v_1^{\epsilon, -})
\, d\sigma_{\epsilon} \\ \nonumber
&\leq \frac{\displaystyle \epsilon}{\displaystyle 2}
\Vert c_1(x,z) \Vert_{L^{\infty}(\Gamma_{\epsilon})} \bigg[
\Vert u_1^{\epsilon, -} \Vert^2_{L^2(\Gamma_{\epsilon})}+
\Vert v_1^{\epsilon, -} \Vert^2_{L^2(\Gamma_{\epsilon})} \bigg] \\ \nonumber
&\leq C_1 \, \epsilon \, \eta \bigg[ \Vert \nabla u_1^{\epsilon, -} 
\Vert^2_{L^2(\Pi_{\epsilon})}+
\Vert \nabla v_1^{\epsilon, -} \Vert^2_{L^2(\Omega_{\epsilon})} \bigg] \\
&+C_2 \, \epsilon \, \eta^{-1} \bigg[ \Vert u_1^{\epsilon, -}
\Vert^2_{L^2(\Pi_{\epsilon})}+
\Vert v_1^{\epsilon, -} \Vert^2_{L^2(\Omega_{\epsilon})} \bigg]
\end{eqnarray}
where $\eta$ is a small constant.
Finally from Eqs. (\ref{2.13}) and (\ref{2.22}) it follows that

\begin{eqnarray} \label{2.23} \nonumber
&\frac{\displaystyle 1}{\displaystyle 2} \partial_t \Vert u_1^{\epsilon, -}
\Vert^2_{L^2(\Pi_{\epsilon})}+
\frac{\displaystyle 1}{\displaystyle 2} \partial_t \Vert v_1^{\epsilon, -}
\Vert^2_{L^2(\Omega_{\epsilon})}+
(\epsilon^2 D_1-C_1 \, \epsilon \, \eta) \, \Vert
\nabla_x u_1^{\epsilon, -} \Vert^2_{L^2(\Pi_{\epsilon})} \\ \nonumber
&+(\tilde{D}_1-C_1 \, \epsilon \, \eta) \, \Vert
\partial_z u_1^{\epsilon, -} \Vert^2_{L^2(\Pi_{\epsilon})}
+(d_1-C_1 \, \epsilon \, \eta) \, \Vert
\nabla v_1^{\epsilon, -} \Vert^2_{L^2(\Omega_{\epsilon})} \\
&\leq C_2 \, \epsilon \, \eta^{-1} \bigg[ 
\Vert u_1^{\epsilon, -} \Vert^2_{L^2(\Pi_{\epsilon})}+
\Vert v_1^{\epsilon, -} \Vert^2_{L^2(\Omega_{\epsilon})} \bigg]
\end{eqnarray}
If one chooses $\eta < min \{\frac{\epsilon \, D_1}{C_1},
\frac{\tilde{D}_1}{\epsilon \, C_1}, 
\frac{d_1}{\epsilon \, C_1} \}$, Eq. (\ref{2.23}) reduces to:

\begin{equation} \label{2.24} 
\frac{\displaystyle 1}{\displaystyle 2} \partial_t \Vert u_1^{\epsilon, -}
\Vert^2_{L^2(\Pi_{\epsilon})}+
\frac{\displaystyle 1}{\displaystyle 2} \partial_t \Vert v_1^{\epsilon, -}
\Vert^2_{L^2(\Omega_{\epsilon})} \leq C_2 \, \epsilon \, \eta^{-1} \bigg[
\Vert u_1^{\epsilon, -} \Vert^2_{L^2(\Pi_{\epsilon})}+
\Vert v_1^{\epsilon, -} \Vert^2_{L^2(\Omega_{\epsilon})} \bigg]
\end{equation}

Setting the initial conditions: $u_1^{\epsilon, -}(0) \equiv 0$ and
$v_1^{\epsilon, -}(0) \equiv 0$, Gronwall's lemma gives:

\begin{equation} \label{2.25}
\Vert u_1^{\epsilon, -} (t, \cdot)\Vert^2_{L^2(\Pi_{\epsilon})}+
\Vert v_1^{\epsilon, -} (t, \cdot)\Vert^2_{L^2(\Omega_{\epsilon})} \leq 0 
\end{equation}
that is, $u_1^{\epsilon} \geq 0$ a.e. in $\Pi_{\epsilon}$ and
$v_1^{\epsilon} \geq 0$ a.e. in $\Omega_{\epsilon}$, for all $t \in [0,T]$.

In the case $1 < i \leq M$, by testing Eqs. (\ref{2.7a}) and (\ref{2.8a})
with $\psi_i=-u_i^{\epsilon, -}(t, \cdot)$ and
$\phi_i=-v_i^{\epsilon, -}(t, \cdot)$, respectively, and decomposing
the functions $u_i^{\epsilon}$ and $v_i^{\epsilon}$ in the positive and
negative parts, one gets:

\begin{eqnarray} \label{add10} \nonumber
&\displaystyle \int_{\Pi_{\epsilon}} \frac{\partial u_i^{\epsilon, -}}
{\partial t} \, u_i^{\epsilon, -} \, dx \, dz
+\epsilon^2 \displaystyle \int_{\Pi_{\epsilon}}
 D_i \, \vert \nabla_x u_i^{\epsilon, -} \vert^2 \, dx \, dz \\ \nonumber
&+\displaystyle \int_{\Pi_{\epsilon}} \tilde{D}_i \, \vert \partial_z
u_i^{\epsilon, -} \vert^2 \, dx \, dz
= \epsilon \displaystyle \int_{\Gamma_{\epsilon}} c_i(x,z) (u_i^{\epsilon}-
v_i^{\epsilon})_{+}
\, u_i^{\epsilon, -} \, d\sigma_{\epsilon} \\ \nonumber
&-\frac{\displaystyle 1}{\displaystyle 2}
\displaystyle \int_{\Pi_{\epsilon}} \bigg[ \sum_{j=1}^{i-1} a_{j,i-j} \,
\sigma_{\tilde M} (u_j^{\epsilon}) \, \sigma_{\tilde M} (u_{i-j}^{\epsilon})
\bigg] \, u_i^{\epsilon, -} \, dx \, dz\\
&+\displaystyle \int_{\Pi_{\epsilon}} \bigg[ \sum_{j=1}^{M} a_{i,j} \,
\sigma_{\tilde M} (u_i^{\epsilon}) \, \sigma_{\tilde M} (u_j^{\epsilon}) 
\bigg] \, u_i^{\epsilon, -} \, dx \, dz
\end{eqnarray}

\begin{eqnarray} \label{add11} \nonumber
&\displaystyle \int_{\Omega_{\epsilon}} \frac{\partial v_i^{\epsilon, -}}
{\partial t} \, v_i^{\epsilon, -} \, dx \, dz
+ \displaystyle \int_{\Omega_{\epsilon}}
 d_i \, \vert \nabla v_i^{\epsilon, -} \vert^2 \, dx \, dz \\\nonumber
&=
-\epsilon \displaystyle \int_{\Gamma_{\epsilon}} c_i(x,z) 
(u_i^{\epsilon}-v_i^{\epsilon})_{+}
\, v_i^{\epsilon, -} \, d\sigma_{\epsilon} \\ \nonumber
&-\frac{\displaystyle 1}{\displaystyle 2}
\displaystyle \int_{\Omega_{\epsilon}} \bigg[ \sum_{j=1}^{i-1} b_{j,i-j} \,
\sigma_{\tilde M} (v_j^{\epsilon}) \, \sigma_{\tilde M} (v_{i-j}^{\epsilon})
\bigg] \, v_i^{\epsilon, -} \, dx \, dz \\
&+\displaystyle \int_{\Omega_{\epsilon}} \bigg[ \sum_{j=1}^{M} b_{i,j} \,
\sigma_{\tilde M} (v_i^{\epsilon}) \, \sigma_{\tilde M} (v_j^{\epsilon}) 
\bigg] \, v_i^{\epsilon, -} \, dx \, dz
\end{eqnarray}
Since in both Eqs. (\ref{add10}) and (\ref{add11}) the last but one term on the
right-hand side is negative and the last one always zero, we obtain that
also the functions $u_i^{\epsilon, -}$ and $v_i^{\epsilon, -}$
($1 < i \leq M$) satisfy Eqs. (\ref{2.10}) and  (\ref{2.12}), respectively.
Therefore, applying exactly the same arguments considered for $i=1$, we
conclude that
$u_i^{\epsilon} \geq 0$ a.e. in $\Pi_{\epsilon}$ and
$v_i^{\epsilon} \geq 0$ a.e. in $\Omega_{\epsilon}$, for all $t \in [0,T]$.

Let us now prove the boundedness of solutions.
In the case $i=1$,
we test Eq. (\ref{2.7a}) with $\psi_1=p (u_1^{\epsilon})^{p-1} (p \geq 2)$:

\begin{eqnarray} \label{2.26a}\nonumber
&p \displaystyle \int_{\Pi_{\epsilon}} \partial_t u_1^{\epsilon} \, 
(u_1^{\epsilon})^{p-1} \, dx \, dz+
\epsilon^2 \, p \displaystyle \int_{\Pi_{\epsilon}} D_1 \,
\nabla_x u_1^{\epsilon} \cdot \nabla_x (u_1^{\epsilon})^{p-1}
\, dx \, dz \\ \nonumber
&+p \displaystyle \int_{\Pi_{\epsilon}} \tilde{D}_1 \, 
\partial_z u_1^{\epsilon} \cdot \partial_z (u_1^{\epsilon})^{p-1}
\, dx \, dz
+\epsilon \, p \displaystyle \int_{\Gamma_{\epsilon}} c_1(x,z) 
(u_1^{\epsilon}-v_1^{\epsilon})_{+} \, (u_1^{\epsilon})^{p-1}   
d\sigma_{\epsilon} \\ \nonumber
&=-p \displaystyle \int_{\Pi_{\epsilon}} 
\bigg[ \sum_{j=1}^{M} a_{1,j} \,
\sigma_{\tilde M} (u_1^{\epsilon}) \, \sigma_{\tilde M} (u_j^{\epsilon}) 
\bigg] \, (u_1^{\epsilon})^{p-1} \, dx \, dz \\
&+p \displaystyle \int_{\Pi_{\epsilon}} f^{\epsilon}(t,x,z) \, 
(u_1^{\epsilon})^{p-1} \, dx \, dz
\end{eqnarray}
Since the last term on the left-hand side is positive
and the first term on the right-hand side is negative, Eq. (\ref{2.26a})
reduces to:

\begin{eqnarray} \label{2.26} \nonumber
&\displaystyle \int_{\Pi_{\epsilon}} \frac{\partial}{\partial t}
(u_1^{\epsilon})^{p} \, dx \, dz+
\epsilon^2 \, p \, (p-1) \displaystyle \int_{\Pi_{\epsilon}} D_1 \,
(u_1^{\epsilon})^{p-2} \, \vert \nabla_x u_1^{\epsilon} \vert^2 
\, dx \, dz \\
&+p \, (p-1) \displaystyle \int_{\Pi_{\epsilon}} \tilde{D}_1 \,
(u_1^{\epsilon})^{p-2} \, \vert \partial_z u_1^{\epsilon} \vert^2 
\, dx \, dz \leq
p \displaystyle \int_{\Pi_{\epsilon}} f^{\epsilon}(t,x,z) \,
(u_1^{\epsilon})^{p-1} \, dx \, dz
\end{eqnarray}
Taking into account that the last two terms on the left-hand side are
positive and integrating over $[0,t]$ with $t \in [0,T]$, we obtain:

\begin{equation} \label{2.27}
\displaystyle \int_0^t \int_{\Pi_{\epsilon}} 
\frac{\partial}{\partial s} (u_1^{\epsilon})^p \, ds \, dx \, dz
\leq p \displaystyle \int_0^t \int_{\Pi_{\epsilon}}
f^{\epsilon}(s,x,z) \, (u_1^{\epsilon})^{p-1} \, ds \, dx \, dz
\end{equation}
Hence,

\begin{equation} \label{addendum1}
\displaystyle \int_{\Pi_{\epsilon}} (u_1^{\epsilon})^p \, dx \, dz
\leq \displaystyle \int_{\Pi_{\epsilon}} (u_1^{\epsilon}(0))^p \, dx \, dz 
+p \, \Vert f^{\epsilon} \Vert_{L^{\infty}([0,T] \times \Pi_{\epsilon})}
\displaystyle \int_0^t \int_{\Pi_{\epsilon}} (u_1^{\epsilon})^{p-1} \, ds \, 
dx \, dz
\end{equation}
In order to estimate the last term on the right-hand side of
Eq. (\ref{addendum1}), it is
now convenient to use Young's inequality in the following form \cite{Brezis}:

\begin{equation} \label{Bineq}
a \, b \leq \eta \, a^{p'}+\eta^{1-p} \, b^p,  \; \; \; \forall a \geq 0, 
\; b \geq 0, \; \; p'=\frac{p}{p-1}
\end{equation}
with: $a=(u_1^{\epsilon})^{p-1}$,
$b=p \, \Vert f^{\epsilon} \Vert_{L^{\infty}([0,T] \times \Pi_{\epsilon})}$
and for an arbitrary $\eta >0$.
Therefore, we get:

\begin{eqnarray} \label{2.27a} \nonumber
&p \, \Vert f^{\epsilon} \Vert_{L^{\infty}([0,T] \times \Pi_{\epsilon})}
\displaystyle \int_0^t \int_{\Pi_{\epsilon}} (u_1^{\epsilon})^{p-1} \, ds \,
dx \, dz  \\ \nonumber
&\leq \displaystyle \int_0^t ds \, \int_{\Pi_{\epsilon}} p^p \, 
\Vert f^{\epsilon} \Vert^p_{L^{\infty}([0,T] \times \Pi_{\epsilon})} \,
\eta^{1-p} \, dx dz  
+\displaystyle \int_0^t ds \, \int_{\Pi_{\epsilon}} \eta \, 
(u_1^{\epsilon})^{p} \, dx \, dz \\
&\leq
p^{p-1} \, \Vert f^{\epsilon} \Vert^p_{L^{\infty}([0,T] \times \Pi_{\epsilon})}
\, \eta^{1-p} \, \vert \Pi_{\epsilon} \vert \, T+ \eta \, 
\displaystyle \int_0^t ds \, \int_{\Pi_{\epsilon}}
(u_1^{\epsilon})^{p} \, dx \, dz
\end{eqnarray}
Choosing $\eta=p$ and using the inequality (\ref{2.27a}) in (\ref{addendum1}),
we conclude that:

\begin{equation} \label{addendum2} 
\Vert u_1^{\epsilon} \Vert^p_{L^p(\Pi_{\epsilon})} \leq 
\Vert U_1^{\epsilon} \Vert^p_{L^p(\Pi_{\epsilon})}+
\Vert f^{\epsilon} \Vert^p_{L^{\infty}([0,T] \times \Pi_{\epsilon})} 
\vert \Pi_{\epsilon} \vert \, T 
+p \displaystyle \int_0^t ds \, 
\Vert u_1^{\epsilon} (s)\Vert^p_{L^p(\Pi_{\epsilon})}
\end{equation}
Applying Gronwall's inequality it follows that:

\begin{equation} \label{2.27b}
\Vert u_1^{\epsilon} \Vert^p_{L^p(\Pi_{\epsilon})} \leq
\bigg[ \Vert U_1^{\epsilon} \Vert^p_{L^p(\Pi_{\epsilon})}+
\Vert f^{\epsilon} \Vert^p_{L^{\infty}([0,T] \times \Pi_{\epsilon})} \, 
\vert \Pi_{\epsilon} \vert  \, T \bigg] \, e^{p t}  
\end{equation}
And finally

\begin{equation} \label{2.27c}
\sup_{t \in [0,T]} \lim_{p \rightarrow \infty} \bigg[
\displaystyle \int_{\Pi_{\epsilon}} (u_1^{\epsilon})^p \, dx \, dz
\bigg]^{1/p} \leq \bigg[ \Vert U_1^{\epsilon} 
\Vert_{L^{\infty}(\Pi_{\epsilon})}+\Vert f^{\epsilon} 
\Vert_{L^{\infty}([0,T] \times \Pi_{\epsilon})} \bigg] \, e^T \leq M_1
\end{equation}
where $M_1$ is a positive constant due to the boundedness of the
initial condition $U_1^{\epsilon}(x,z)$ and of the source term
$f^{\epsilon} (t,x,z)$.

We test now Eq. (\ref{2.8a}) with
$\phi_1=(v_1^{\epsilon}(t, \cdot)-\overline {M}_1)^+$, where $\overline {M}_1$
is a positive constant:

\begin{eqnarray}  \nonumber
&\displaystyle \int_{\Omega_{\epsilon}} \partial_t (v_1^{\epsilon}-
\overline {M}_1) \, (v_1^{\epsilon}-\overline {M}_1)^+ \, dx \, dz+
\displaystyle \int_{\Omega_{\epsilon}} d_1 
\nabla (v_1^{\epsilon}-\overline {M}_1) \cdot 
\nabla (v_1^{\epsilon}-\overline {M}_1)^+ \, dx \, dz \\ \nonumber
&-\epsilon \displaystyle \int_{\Gamma_{\epsilon}} c_1(x,z) 
(u_1^{\epsilon}-v_1^{\epsilon})_{+} \, (v_1^{\epsilon}-\overline {M}_1)^+
d\sigma_{\epsilon} \\
&=\displaystyle \int_{\Omega_{\epsilon}} N_1^{\tilde M}(v^{\epsilon}) \, 
(v_1^{\epsilon}-\overline {M}_1)^+ \, dx \, dz
\end{eqnarray}
Decomposing the function $(v_1^{\epsilon}-\overline {M}_1)$ in its positive
and negative parts, one obtains:

\begin{eqnarray} \label{2.28} \nonumber
&\frac{\displaystyle 1}{\displaystyle 2} \displaystyle \int_{\Omega_{\epsilon}}
\partial_t \vert (v_1^{\epsilon}-\overline {M}_1)^+ \vert^2 
\, dx \, dz+
 \displaystyle \int_{\Omega_{\epsilon}} d_1
\vert \nabla (v_1^{\epsilon}-\overline {M}_1)^+ \vert^2  \, dx \, dz 
\\ \nonumber
&=\epsilon \displaystyle \int_{\Gamma_{\epsilon}} c_1(x,z)
(u_1^{\epsilon}-v_1^{\epsilon})_{+} \, (v_1^{\epsilon}-\overline {M}_1)^+ 
d\sigma_{\epsilon} \\
&-\displaystyle \int_{\Omega_{\epsilon}}
\bigg[ \sum_{j=1}^{M} b_{1,j} \,
\sigma_{\tilde M} (v_1^{\epsilon}) \, \sigma_{\tilde M} (v_j^{\epsilon}) 
\bigg] \, (v_1^{\epsilon}-\overline {M}_1)^+ \, dx \, dz
\end{eqnarray}
Since the last term on the right-hand side is negative, we conclude:

\begin{eqnarray} \label{2.29} \nonumber
&\frac{\displaystyle 1}{\displaystyle 2} \partial_t \Vert
(v_1^{\epsilon}-\overline {M}_1)^+ \Vert^2_{L^2(\Omega_{\epsilon})}+
d_1 \Vert \nabla (v_1^{\epsilon}-\overline {M}_1)^+ 
\Vert^2_{L^2(\Omega_{\epsilon})} \\
&\leq \epsilon \displaystyle \int_{\Gamma_{\epsilon}} c_1(x,z)
(u_1^{\epsilon}-v_1^{\epsilon})_{+} \,
(v_1^{\epsilon}-\overline {M}_1)^+ d\sigma_{\epsilon}
\end{eqnarray}
Let us now estimate the term on the right-hand side of (\ref{2.29}):

\begin{eqnarray} \label{2.31} \nonumber
&\epsilon \displaystyle \int_{\Gamma_{\epsilon}} c_1(x,z)
(u_1^{\epsilon}-v_1^{\epsilon})_{+} \,
(v_1^{\epsilon}-\overline {M}_1)^+ d\sigma_{\epsilon} \\ \nonumber
&\leq
\epsilon \displaystyle \int_{\Gamma_{\epsilon}} c_1(x,z)
{\mathcal H}(u_1^{\epsilon}-v_1^{\epsilon}) \, (u_1^{\epsilon}-M_1) \,
(v_1^{\epsilon}-\overline {M}_1)^+ d\sigma_{\epsilon} \\ \nonumber
&-\epsilon \displaystyle \int_{\Gamma_{\epsilon}} c_1(x,z)
{\mathcal H}(u_1^{\epsilon}-v_1^{\epsilon}) \, (v_1^{\epsilon}-M_1) \,
(v_1^{\epsilon}-\overline {M}_1)^+ d\sigma_{\epsilon}\\
&\leq \epsilon \displaystyle \int_{\Gamma_{\epsilon}} c_1(x,z)
{\mathcal H}(u_1^{\epsilon}-v_1^{\epsilon}) \, (u_1^{\epsilon}-M_1)^+ \,
(v_1^{\epsilon}-\overline {M}_1)^+ d\sigma_{\epsilon}
\end{eqnarray}
if one chooses $M_1 \leq \overline {M}_1$.
By applying the H\"older and Young inequalities, along with (\ref{A.7}),
one gets:

\begin{eqnarray} \label{2.32} \nonumber
&\epsilon \displaystyle \int_{\Gamma_{\epsilon}} c_1(x,z)
(u_1^{\epsilon}-v_1^{\epsilon})_{+} \,
(v_1^{\epsilon}-\overline {M}_1)^+ d\sigma_{\epsilon} \\ \nonumber
&\leq
\frac{\displaystyle \epsilon}{\displaystyle 2} \Vert c_1(x,z) 
\Vert_{L^{\infty}(\Gamma_{\epsilon})} 
\bigg[ \Vert (u_1^{\epsilon}-M_1)^+ \Vert^2_{L^2(\Gamma_{\epsilon})}+
\Vert (v_1^{\epsilon}-\overline {M}_1)^+ \Vert^2_{L^2(\Gamma_{\epsilon})}
\bigg] \\ 
&\leq C_1 \, \epsilon \, \eta \,
\Vert \nabla (v_1^{\epsilon}-\overline {M}_1)^+ 
\Vert^2_{L^2(\Omega{\epsilon})}  
+C_2 \, \epsilon \, \eta^{-1} 
\Vert (v_1^{\epsilon}-\overline {M}_1)^+
\Vert^2_{L^2(\Omega{\epsilon})} 
\end{eqnarray}
where $\eta$ is a small constant and $u_1^{\epsilon}$ satisfies the
inequality (\ref{2.27c}).
Therefore, from Eq. (\ref{2.29}):

\begin{eqnarray} \label{2.33} \nonumber
&\frac{\displaystyle 1}{\displaystyle 2} \partial_t \Vert
(v_1^{\epsilon}-\overline {M}_1)^+ \Vert^2_{L^2(\Omega_{\epsilon})} 
+(d_1-C_1 \epsilon \eta)
\Vert \nabla (v_1^{\epsilon}-\overline {M}_1)^+
\Vert^2_{L^2(\Omega_{\epsilon})} \\
&\leq C_2 \, \epsilon \, \eta^{-1}  
\, \Vert (v_1^{\epsilon}-\overline {M}_1)^+ \Vert^2_{L^2(\Omega_{\epsilon})}
\end{eqnarray}
If we choose
$\eta < 
\frac{d_1}{\epsilon \, C_1}$, Eq. (\ref{2.33}) reduces to:

\begin{equation} \label{2.34} 
\frac{\displaystyle 1}{\displaystyle 2} \partial_t \Vert
(v_1^{\epsilon}-\overline {M}_1)^+ \Vert^2_{L^2(\Omega_{\epsilon})} 
\leq C_2 \, \epsilon \, \eta^{-1}  
\, \Vert (v_1^{\epsilon}-\overline {M}_1)^+ \Vert^2_{L^2(\Omega_{\epsilon})}
\end{equation}
and the Gronwall lemma gives:

\begin{equation} \label{2.35} 
\Vert (v_1^{\epsilon}(t, \cdot)-\overline {M}_1)^+ 
\Vert^2_{L^2(\Omega_{\epsilon})}
\leq
\Vert (v_1^{\epsilon}-\overline {M}_1)^+ (0) \Vert^2_{L^2(\Omega_{\epsilon})}
\, \exp (2 \, C_2 \, \epsilon \, \eta^{-1} \, T)
\end{equation}
for all $t \in [0,T]$.
Since $v_1^{\epsilon}(0)=0$, then:

$$\Vert (v_1^{\epsilon}-\overline {M}_1)^+ (0) \Vert^2_{L^2(\Omega_{\epsilon})}
=0$$
and from Eq. (\ref{2.35}) it follows:

\begin{equation} \label{2.35a}
v_1^{\epsilon} \leq \overline {M}_1.
\end{equation}

In the case $1 < i \leq M$ we proceed by induction and
test Eq. (\ref{2.7a}) with
$\psi_i=p (u_i^{\epsilon})^{p-1} (p \geq 2)$:

\begin{eqnarray} \label{10.1}\nonumber
&p \displaystyle \int_{\Pi_{\epsilon}} \partial_t u_i^{\epsilon} \, 
(u_i^{\epsilon})^{p-1} \, dx \, dz+
\epsilon^2 \, p \displaystyle \int_{\Pi_{\epsilon}} D_i \,
\nabla_x u_i^{\epsilon} \cdot \nabla_x (u_i^{\epsilon})^{p-1}
\, dx \, dz \\ \nonumber
&+p \displaystyle \int_{\Pi_{\epsilon}} \tilde{D}_i \, 
\partial_z u_i^{\epsilon} \cdot \partial_z (u_i^{\epsilon})^{p-1}
\, dx \, dz
+\epsilon \, p \displaystyle \int_{\Gamma_{\epsilon}} c_i(x,z) 
(u_i^{\epsilon}-v_i^{\epsilon})_{+} \, (u_i^{\epsilon})^{p-1}   
d\sigma_{\epsilon} \\ \nonumber
&=-p \displaystyle \int_{\Pi_{\epsilon}} 
\bigg[ \sum_{j=1}^{M} a_{i,j} \,
\sigma_{\tilde M} (u_i^{\epsilon}) \, \sigma_{\tilde M} (u_j^{\epsilon}) 
\bigg] \, (u_i^{\epsilon})^{p-1} \, dx \, dz\\
&+\frac{\displaystyle p}{\displaystyle 2}  \displaystyle \int_{\Pi_{\epsilon}} 
\bigg[ \sum_{j=1}^{i-1} a_{j,i-j} \,
\sigma_{\tilde M} (u_j^{\epsilon}) \, \sigma_{\tilde M} (u_{i-j}^{\epsilon})
\bigg] \,
(u_i^{\epsilon})^{p-1} \, dx \, dz
\end{eqnarray}
Since the last term on the left-hand side is positive
and the first term on the right-hand side is negative, Eq. (\ref{10.1})
reduces to:

\begin{eqnarray} \label{10.2} \nonumber
&\displaystyle \int_{\Pi_{\epsilon}} \frac{\partial}{\partial t}
(u_i^{\epsilon})^{p}\, dx \, dz+
\epsilon^2 \, p \, (p-1) \displaystyle \int_{\Pi_{\epsilon}} D_i \,
(u_i^{\epsilon})^{p-2} \, \vert \nabla_x u_i^{\epsilon} \vert^2 
\, dx \, dz \\ \nonumber
&+p \, (p-1) \displaystyle \int_{\Pi_{\epsilon}} \tilde{D}_i \,
(u_i^{\epsilon})^{p-2} \, \vert \partial_z u_i^{\epsilon} \vert^2 
\, dx \, dz\\
&\leq
\frac{\displaystyle p}{\displaystyle 2}  \displaystyle \int_{\Pi_{\epsilon}}
\bigg[ \sum_{j=1}^{i-1} a_{j,i-j} \,
\sigma_{\tilde M} (u_j^{\epsilon}) \, \sigma_{\tilde M} (u_{i-j}^{\epsilon})
\bigg] \,
(u_i^{\epsilon})^{p-1} \, dx \, dz
\end{eqnarray}
Taking into account that the last two terms on the left-hand side are
positive and integrating over $[0,t]$ with $t \in [0,T]$, we obtain:

\begin{equation} \label{10.3}
\displaystyle \int_0^t \int_{\Pi_{\epsilon}} 
\frac{\partial}{\partial s} (u_i^{\epsilon})^p \, ds \, dx \, dz
\leq
\frac{\displaystyle p}{\displaystyle 2} \displaystyle \int_0^t
\int_{\Pi_{\epsilon}} \bigg[ \sum_{j=1}^{i-1} a_{j,i-j} \,
\sigma_{\tilde M} (u_j^{\epsilon}) \, \sigma_{\tilde M} (u_{i-j}^{\epsilon})
\bigg] \, (u_i^{\epsilon})^{p-1} \, ds \, dx \, dz
\end{equation}
Exploiting the boundedness of $u_j^{\epsilon} (1 \leq j \leq i-1)$ in
$L^{\infty}(0,T; L^{\infty} (\Pi_{\epsilon}))$ and setting the initial
conditions, one gets:

\begin{equation} \label{10.4}
\displaystyle \int_{\Pi_{\epsilon}} (u_i^{\epsilon})^p \, dx \, dz
\leq \frac{\displaystyle p}{\displaystyle 2} \displaystyle \int_0^t
\int_{\Pi_{\epsilon}} \bigg[ \sum_{j=1}^{i-1} a_{j,i-j} \,
K_j \, K_{i-j} \bigg] \, (u_i^{\epsilon})^{p-1} \, ds \, dx \, dz 
\end{equation}
where $K_j (1 \leq j \leq i-1)$ are positive constants.
In order to estimate the term on the right-hand side of Eq. (\ref{10.4}),
we use the Young inequality (\ref{Bineq}) with: $a=(u_i^{\epsilon})^{p-1}$
and $b=p \, [\sum_{j=1}^{i-1} a_{j,i-j} \, K_j \, K_{i-j}]$.
We find:

\begin{eqnarray} \label{10.5} \nonumber
&p \, \displaystyle \int_0^t \int_{\Pi_{\epsilon}} \bigg[ \sum_{j=1}^{i-1} 
a_{j,i-j} \, K_j \, K_{i-j} \bigg] \, (u_i^{\epsilon})^{p-1} \, ds \, dx \, dz
\\ \nonumber
&\leq \displaystyle \int_0^t ds \, \int_{\Pi_{\epsilon}} p^p \, 
\bigg[ \sum_{j=1}^{i-1} a_{j,i-j} \, K_j \, K_{i-j} \bigg]^p \,
\eta^{1-p} \, dx dz  
+\displaystyle \int_0^t ds \, \int_{\Pi_{\epsilon}} \eta \, 
(u_i^{\epsilon})^{p} \, dx \, dz \\
&\leq
p^{p-1} \, \bigg[ \sum_{j=1}^{i-1} a_{j,i-j} \, K_j \, K_{i-j} \bigg]^p  
\, \eta^{1-p} \, \vert \Pi_{\epsilon} \vert \, T+ \eta \, 
\displaystyle \int_0^t ds \, \int_{\Pi_{\epsilon}}
(u_i^{\epsilon})^{p} \, dx \, dz
\end{eqnarray}
Choosing $\eta=p$ and using the inequality (\ref{10.5}) in (\ref{10.4}),
we obtain:

\begin{equation} \label{10.6} 
\Vert u_i^{\epsilon} \Vert^p_{L^p(\Pi_{\epsilon})} \leq 
\bigg[ \sum_{j=1}^{i-1} a_{j,i-j} \, K_j \, K_{i-j} \bigg]^p \,
\vert \Pi_{\epsilon} \vert \, T 
+p \displaystyle \int_0^t ds \, 
\Vert u_i^{\epsilon} (s)\Vert^p_{L^p(\Pi_{\epsilon})}
\end{equation}
The Gronwall lemma applied to (\ref{10.6}) leads to the estimate:

\begin{equation} \label{10.7}
\Vert u_i^{\epsilon} \Vert^p_{L^p(\Pi_{\epsilon})} \leq
\bigg[ \sum_{j=1}^{i-1} a_{j,i-j} \, K_j \, K_{i-j} \bigg]^p \,
\vert \Pi_{\epsilon} \vert  \, T \, e^{p t}  
\end{equation}
Hence,

\begin{equation} \label{10.8}
\sup_{t \in [0,T]} \lim_{p \rightarrow \infty} \bigg[
\displaystyle \int_{\Pi_{\epsilon}} (u_i^{\epsilon})^p \, dx \, dz
\bigg]^{1/p} \leq 
\sum_{j=1}^{i-1} a_{j,i-j} \, K_j \, K_{i-j} \,  e^T \leq M_i
\end{equation}
where $M_i$ is a positive constant.

We test now Eq. (\ref{2.8a}) with $\phi_i=p [(v_i^{\epsilon}(t, \cdot)-
\overline {m}_i)^+]^{p-1}$ ($p \geq 2$), where $\overline {m}_i$
is a positive constant:

\begin{eqnarray}  \nonumber
&p \, \displaystyle \int_{\Omega_{\epsilon}} \partial_t (v_i^{\epsilon}-
\overline {m}_i) \, [(v_i^{\epsilon}-\overline {m}_i)^+]^{p-1} 
\, dx \, dz \\ \nonumber
&+p \, \displaystyle \int_{\Omega_{\epsilon}} d_i 
\nabla (v_i^{\epsilon}-\overline {m}_i) \cdot 
\nabla [(v_i^{\epsilon}-\overline {m}_i)^+]^{p-1} \, dx \, dz
\\ \nonumber
&-\epsilon \, p \displaystyle \int_{\Gamma_{\epsilon}} c_i(x,z) 
(u_i^{\epsilon}-v_i^{\epsilon})_{+} \, [(v_i^{\epsilon}-
\overline {m}_i)^+]^{p-1} d\sigma_{\epsilon} \\ \nonumber  
&=-p \displaystyle \int_{\Omega_{\epsilon}} 
\bigg[ \sum_{j=1}^{M} b_{i,j} \,
\sigma_{\tilde M} (v_i^{\epsilon}) \, \sigma_{\tilde M} (v_j^{\epsilon}) 
\bigg] \, [(v_i^{\epsilon}-\overline {m}_i)^+]^{p-1} \, dx \, dz\\
&+\frac{\displaystyle p}{\displaystyle 2}  
\displaystyle \int_{\Omega_{\epsilon}} 
\bigg[ \sum_{j=1}^{i-1} b_{j,i-j} \,
\sigma_{\tilde M} (v_j^{\epsilon}) \, \sigma_{\tilde M} (v_{i-j}^{\epsilon})
\bigg] \,
[(v_i^{\epsilon}-\overline {m}_i)^+]^{p-1} \, dx \, dz
\end{eqnarray}
Decomposing the function $(v_i^{\epsilon}-\overline {m}_i)$
in its positive and negative parts and taking into account the boundedness of
$v_j^{\epsilon} \, (1 \leq j \leq i-1)$ in
$L^\infty (0,T; L^\infty(\Omega_{\epsilon}))$ due to the inductive hypothesis,
one obtains:

\begin{eqnarray} \label{10.9} \nonumber
& \displaystyle \int_{\Omega_{\epsilon}}
\partial_t [(v_i^{\epsilon}-\overline {m}_i)^+]^p \, dx \, dz+
p \, (p-1) \displaystyle \int_{\Omega_{\epsilon}} d_i \,
\vert \nabla (v_i^{\epsilon}-\overline {m}_i)^+ \vert^2 \,
[(v_i^{\epsilon}-\overline {m}_i)^+]^{p-2}  \, dx \, dz \\ \nonumber
& \leq \epsilon \, p \displaystyle \int_{\Gamma_{\epsilon}} c_i(x,z)
(u_i^{\epsilon}-v_i^{\epsilon})_{+} \, 
[(v_i^{\epsilon}-\overline {m}_i)^+]^{p-1} d\sigma_{\epsilon} \\ \nonumber
&+\frac{\displaystyle p}{\displaystyle 2}  
\displaystyle \int_{\Omega_{\epsilon}} 
\bigg[ \sum_{j=1}^{i-1} b_{j,i-j} \,
K_j \, K_{i-j} \bigg]\, [(v_i^{\epsilon}-\overline {m}_i)^+]^{p-1} 
\, dx \, dz \\
&-p \displaystyle \int_{\Omega_{\epsilon}}
\bigg[ \sum_{j=1}^{M} b_{i,j} \,
\sigma_{\tilde M} (v_i^{\epsilon}) \, \sigma_{\tilde M} (v_j^{\epsilon})
\bigg] \, [(v_i^{\epsilon}-\overline {m}_i)^+]^{p-1} \, dx \, dz
\end{eqnarray}
where $K_j (1 \leq j \leq i-1)$ are positive constants.
Since the last term on the left-hand side is positive and the one
on the right-hand side is negative, Eq. (\ref{10.9}) reduces to:

\begin{eqnarray} \label{10.10} \nonumber
& \displaystyle \int_{\Omega_{\epsilon}}
\partial_t [(v_i^{\epsilon}-\overline {m}_i)^+]^p \, dx \, dz \leq
\epsilon \, p \displaystyle \int_{\Gamma_{\epsilon}} c_i(x,z)
(u_i^{\epsilon}-v_i^{\epsilon})_{+} \,
[(v_i^{\epsilon}-\overline {m}_i)^+]^{p-1} d\sigma_{\epsilon} \\
&+\frac{\displaystyle p}{\displaystyle 2} 
\displaystyle \int_{\Omega_{\epsilon}}
\bigg[ \sum_{j=1}^{i-1} b_{j,i-j} \,
K_j \, K_{i-j} \bigg]\, [(v_i^{\epsilon}-\overline {m}_i)^+]^{p-1}
\, dx \, dz
\end{eqnarray}
Let us now estimate the first term, $I_1$, on the right-hand side of
Eq. (\ref{10.10}):

\begin{eqnarray} \label{10.11} \nonumber
&I_1 = \epsilon \, p \displaystyle \int_{\Gamma_{\epsilon}} c_i(x,z)
{\mathcal H} (u_i^{\epsilon}-v_i^{\epsilon}) \, (u_i^{\epsilon}-v_i^{\epsilon})
[(v_i^{\epsilon}-\overline {m}_i)^+]^{p-1} d\sigma_{\epsilon} \\ \nonumber
&\leq \epsilon \, p \displaystyle \int_{\Gamma_{\epsilon}} c_i(x,z)
{\mathcal H} (u_i^{\epsilon}-v_i^{\epsilon}) \, 
(u_i^{\epsilon}-\overline {m}_i) \, 
[(v_i^{\epsilon}-\overline {m}_i)^+]^{p-1} d\sigma_{\epsilon} \\ \nonumber
&-\epsilon \, p \displaystyle \int_{\Gamma_{\epsilon}} c_i(x,z)
{\mathcal H} (u_i^{\epsilon}-v_i^{\epsilon}) \,
(v_i^{\epsilon}-\overline {m}_i) \,
[(v_i^{\epsilon}-\overline {m}_i)^+]^{p-1} d\sigma_{\epsilon} \\
&\leq \epsilon \, p \displaystyle \int_{\Gamma_{\epsilon}} c_i(x,z)
(u_i^{\epsilon}-\overline {m}_i)^+ \, 
[(v_i^{\epsilon}-\overline {m}_i)^+]^{p-1} d\sigma_{\epsilon}
\end{eqnarray}
where the last step in inequality (\ref{10.11} ) has been obtained
decomposing the functions $(u_i^{\epsilon}-\overline {m}_i)$ and
$(v_i^{\epsilon}-\overline {m}_i)$ in positive and negative parts.
By applying the H\"older inequality and exploiting the generalized
interpolation-trace inequality (\ref{A.7}) along with the
the boundedness of $u_i^{\epsilon}$ in
$L^{\infty}(0,T; L^{\infty} (\Pi_{\epsilon}))$, one concludes that:

\begin{equation} \label{10.12}
0 < I_1 \leq \epsilon \, p \, \Vert c_i(x,z) 
\Vert_{L^\infty(\Gamma_{\epsilon})} \, \Vert (u_i^{\epsilon}-\overline {m}_i)^+
\Vert_{L^2(\Gamma_{\epsilon})} \, 
\Vert [(v_i^{\epsilon}-\overline {m}_i)^+]^{p-1}
\Vert_{L^2(\Gamma_{\epsilon})}=0
\end{equation}
Hence, $I_1=0$.

Therefore, Eq. (\ref{10.10}) reduces to:

\begin{equation} \label{10.13} 
 \displaystyle \int_{\Omega_{\epsilon}}
\partial_t [(v_i^{\epsilon}-\overline {m}_i)^+]^p \, dx \, dz \leq
p \, \displaystyle \int_{\Omega_{\epsilon}}
\bigg[ \sum_{j=1}^{i-1} b_{j,i-j} \,
K_j \, K_{i-j} \bigg]\, [(v_i^{\epsilon}-\overline {m}_i)^+]^{p-1}
\, dx \, dz
\end{equation}
Integrating over $[0,t]$ with $t \in [0,T]$, and estimating the term
on the right-hand side by using Young's inequality in the form (\ref{Bineq}),
with: $a=[(v_i^{\epsilon}-\overline {m}_i)^+]^{p-1}$ and
$b=p \, \bigg[ \sum_{j=1}^{i-1} b_{j,i-j} \, K_j \, K_{i-j} \bigg]$, we get

\begin{eqnarray}  \nonumber
&\displaystyle \int_0^t \int_{\Omega_{\epsilon}}
\frac{\partial}{\partial s} [(v_i^{\epsilon}-\overline {m}_i)^+]^p
\, ds \, dx \, dz \\ \nonumber
& \leq p^{p-1} \, \bigg[ \sum_{j=1}^{i-1} b_{j,i-j} \,
K_j \, K_{i-j} \bigg]^p \, \eta^{1-p} \, \vert \Omega_{\epsilon} \vert \, T+
\eta \, \displaystyle \int_0^t \int_{\Omega_{\epsilon}}
[(v_i^{\epsilon}-\overline {m}_i)^+]^p \, ds \, dx \, dz
\end{eqnarray}
Choosing $\eta=p$ and setting the initial conditions, it follows that:

\begin{eqnarray} \label{10.14} \nonumber
& \Vert (v_i^{\epsilon}-\overline {m}_i)^+ \Vert^p_{L^p(\Omega_{\epsilon})}
\leq \bigg[ \sum_{j=1}^{i-1} b_{j,i-j} \,
K_j \, K_{i-j} \bigg]^p \, \vert \Omega_{\epsilon} \vert \, T \\
&+p \, \displaystyle \int_0^t ds \, 
\Vert (v_i^{\epsilon}-\overline {m}_i)^+ \Vert^p_{L^p(\Omega_{\epsilon})}
\end{eqnarray}
Finally, the Gronwall Lemma leads to the estimate

\begin{equation} \label{10.15} 
 \Vert (v_i^{\epsilon}-\overline {m}_i)^+ \Vert^p_{L^p(\Omega_{\epsilon})}
\leq \bigg[ \sum_{j=1}^{i-1} b_{j,i-j} \,
K_j \, K_{i-j} \bigg]^p \, \vert \Omega_{\epsilon} \vert \, T \, e^{p \, t}
\end{equation}
Hence,

\begin{equation} \label{10.16}
\sup_{t \in [0,T]} \lim_{p \rightarrow \infty} \bigg[
\displaystyle \int_{\Omega_{\epsilon}}
[(v_i^{\epsilon}-\overline {m}_i)^+]^p \, dx\, dz \bigg]^{1/p} \leq
\bigg[ \sum_{j=1}^{i-1} b_{j,i-j} \,
K_j \, K_{i-j} \bigg] \, e^T 
\end{equation}
Therefore, since the positive part of the function
$(v_i^{\epsilon}-\overline {m}_i)$ is bounded,
it follows that: $v_i^{\epsilon}(t,x,z) \leq \overline{M}_i$, where
$\overline{M}_i$ is a positive constant.

\end{proof}

\begin{lemma} \label{l3.2}

For a given small $\epsilon >0$,
let $u_i^{\epsilon}(t,x,z)$ and $v_i^{\epsilon}(t,x,z)$ ($1 \leq i \leq M$)
be solutions of the system
(\ref{2.1a}) in the sense of the Definition \ref{d2.1}.
Then, the following estimates hold:

\begin{equation} \label{11.1}
\epsilon \Vert \nabla_x u_i^{\epsilon} 
\Vert_{L^{\infty} (0,T; L^2(\Pi_{\epsilon}))} \leq C_i^x
\end{equation}

\begin{equation} \label{11.2}
\Vert \partial_z u_i^{\epsilon} 
\Vert_{L^{\infty} (0,T; L^2(\Pi_{\epsilon}))} \leq C_i^z
\end{equation}

\begin{equation} \label{11.3}
\Vert \partial_t u_i^{\epsilon} 
\Vert_{L^2 (0,T; L^2(\Pi_{\epsilon}))} \leq C_i^t
\end{equation}

\begin{equation} \label{11.4}
\Vert \nabla v_i^{\epsilon} 
\Vert_{L^{\infty} (0,T; L^2(\Omega_{\epsilon}))} \leq D_i
\end{equation}

\begin{equation} \label{11.5}
\Vert \partial_t v_i^{\epsilon} 
\Vert_{L^2 (0,T; L^2(\Omega_{\epsilon}))} \leq D_i^t
\end{equation}
where $1 \leq i \leq M$ and $C_i^x, C_i^z, C_i^t, D_i, D_i^t$
are positive constants independent of $\tilde M$ and $\epsilon$.

\end{lemma}

\begin{proof}
In the case $i=1$,
let $u_{1,n}^{\epsilon}$ and $v_{1,n}^{\epsilon}$ be the approximate solutions
defined in the proof of Lemma \ref{l4.1}.
Then the inequality (\ref{4.35}) holds with a constant
$\tilde{C} \geq 0$ independent
of $n$ and $\epsilon$.
Using the convergence results reported in Proposition \ref{conv} and the
lower-semicontinuity of the norm from Theorem \ref{tB.2} \cite{CD} (see
Appendix B), we get from Eq. (\ref{4.35}):

\begin{eqnarray} \label{11.6} \nonumber
& \Vert  \partial_t u_1^{\epsilon}
\Vert^2_{L^2(0,T; L^2(\Pi_{\epsilon}))}+
2 \,\Vert \partial_t v_1^{\epsilon}
\Vert^2_{L^2(0,T; L^2(\Omega_{\epsilon}))}+
2 \,\epsilon^2 \,  D_1\, 
\Vert \nabla_x u_1^{\epsilon} \Vert^2_{L^{\infty}(0,T; L^2(\Pi_{\epsilon}))} \\
\nonumber
&+2 \, \tilde{D}_1 \, \Vert \partial_z u_1^{\epsilon} \Vert^2_{L^{\infty}
(0,T; L^2(\Pi_{\epsilon}))}
+2 \,d_1 \,
\Vert \nabla v_1^{\epsilon} \Vert^2_{L^{\infty}(0,T; L^2(\Omega_{\epsilon}))} 
\\ \nonumber
& \leq \lim \inf_{n \rightarrow \infty}
\Vert  \partial_t u_{1,n}^{\epsilon}
\Vert^2_{L^2(0,T; L^2(\Pi_{\epsilon}))} 
+\lim \inf_{n \rightarrow \infty}
2 \,\Vert \partial_t v_{1,n}^{\epsilon}
\Vert^2_{L^2(0,T; L^2(\Omega_{\epsilon}))} \\ \nonumber
&+ \lim \inf_{n \rightarrow \infty}
2 \, \epsilon^2 \,  D_1 \, 
\Vert \nabla_x u_{1,n}^{\epsilon} \Vert^2_{L^{\infty}(0,T; L^2(\Pi_{\epsilon}))}
+\lim \inf_{n \rightarrow \infty}
2 \, \tilde{D}_1 \, \Vert \partial_z u_{1,n}^{\epsilon} 
\Vert^2_{L^{\infty}(0,T; L^2(\Pi_{\epsilon}))} \\
&+\lim \inf_{n \rightarrow \infty}
2 \,d_1 \,
\Vert \nabla v_{1,n}^{\epsilon} \Vert^2_{L^{\infty}(0,T; 
L^2(\Omega_{\epsilon}))} \leq \tilde{C}
\end{eqnarray}

For the case $1 < i \leq M$, the proof carries over verbatim, since also the
functions $u_{i,n}^{\epsilon}$ and $v_{i,n}^{\epsilon}$ ($1 < i \leq M$)
satisfy the inequality (\ref{4.35}).

\end{proof}

\section{Homogenization}\label{sec4}

The behavior of the solutions $u_i^{\epsilon}, v_i^{\epsilon}$
$(1 \leq i \leq M)$ of the set of Eqs. (\ref{2.1a})
as $\epsilon \rightarrow 0$ will now be studied.
In order to pass to the limit, it is necessary to obtain equations
and estimates in $\Omega$.

\begin{lemma} \label{l6.2}

Let us consider the sets defined in Section \ref{sec1}.

\par\noindent
1. There exists a linear continuous extension operator

\begin{equation} \label{6.5}
{\tilde P}: H^1(X) \rightarrow H^1(Y)
\end{equation}
such that

\begin{equation}
{\tilde P} u=u \quad \text{in} \; X
\end{equation}
and

\begin{equation} \label{6.6a}
\Vert {\tilde P} \, u \Vert_{L^2(Y)} \leq C \Vert 
u \Vert_{L^2(X)}
\end{equation}

\begin{equation} \label{6.6}
\Vert \nabla ({\tilde P}\, u) \Vert_{L^2(Y)} \leq C \Vert \nabla 
u \Vert_{L^2(X)}
\end{equation}
where $C$ is a positive constant.

\par\noindent
2. There exists a family of linear continuous extension operators

\begin{equation} \label{6.7}
{\tilde P}_{\epsilon}: H^1(G_{\epsilon}) \rightarrow H^1(D)
\end{equation}
such that

\begin{equation}
{\tilde P}_{\epsilon} u^{\epsilon}=u^{\epsilon} \quad \text{in} \; G_{\epsilon}
\end{equation}
and

\begin{equation} \label{6.8a}
\Vert {\tilde P}_{\epsilon} u^{\epsilon} \Vert_{L^2(D)} \leq C 
\Vert u^{\epsilon} \Vert_{L^2(G_{\epsilon})}
\end{equation}

\begin{equation} \label{6.8}
\Vert \nabla ({\tilde P}_{\epsilon} u^{\epsilon}) \Vert_{L^2(D)} \leq C 
\Vert \nabla u^{\epsilon} \Vert_{L^2(G_{\epsilon})}
\end{equation}
where the constant $C>0$ does not depend on $\epsilon$.

\end{lemma}

\begin{proof}

1. First we extend $u$ into a neighbourhood $X_0$ of $R=\partial X$
with smooth boundary,
such that $\text{clos} (X) \subset X_0$ and $\text{clos} (X_0) \subset 
\text{int} (Y)$.
Since we have assumed that $R$ is sufficiently smooth, we can construct
a diffeomorphism as follows:

\begin{eqnarray} \label{diffeo}
&\Phi: R \times ]-\delta, \delta[ \rightarrow X_0 \\
&\Phi(y, \lambda)=x
\end{eqnarray}
Exploiting this coordinate transformation, the function $u$ can be extended
by reflection:

\begin{equation} \label{reflect}
u^*(x)=u^*(\Phi(y, \lambda))=
\begin{cases}
u(\Phi(y, \lambda)) & \lambda \ge 0 \\
u(\Phi(y, -\lambda)) & \lambda <0
\end{cases}
\end{equation}
Let us now consider the following smooth function

\begin{equation} \label{func}
\Psi: Y \rightarrow [0,1]
\end{equation}
such that $\text{supp} \Psi \subseteq Z$ and $\Psi=1$ in $Z \setminus X_0$.
Then, we define

\begin{equation} \label{extens}
\tilde{u}(x):=(1-\Psi)(u^*(x)-m)+m
\end{equation}
where
\begin{equation} \label{media}
m:=\frac{\displaystyle 1}{\displaystyle \vert X \vert}
\displaystyle \int_X u(y) \, dy
\end{equation}

Let us prove that $\tilde{u}(x)$ is an extension of $u(x)$, which satisfies
(\ref{6.6a}) and (\ref{6.6}).
One gets:

\begin{eqnarray} \label{ineq1} \nonumber
&\Vert \tilde{u} \Vert^2_{L^2(Z)}=\displaystyle \int_Z
\vert (1-\Psi) (u^*-m)+m \vert^2 \, dx \\  
&=\displaystyle \int_{Z \cap X_0} \vert (1-\Psi) u^*+\Psi m \vert^2 \, dx
+\displaystyle \int_{Z \setminus X_0} m^2 \, dx
\end{eqnarray}
Taking into account the following estimates

\begin{equation} \label{ineq2} 
\displaystyle \int_{Z \cap X_0} \bigg[ \frac{\displaystyle 1}{\displaystyle 
\vert X \vert} \displaystyle \int_X u(y) \, dy \bigg]^2 \, dx
\leq \displaystyle \int_{Z \cap X_0} \frac{\displaystyle 1}{\displaystyle
\vert X \vert} \bigg( \displaystyle \int_X u^2 \, dy \bigg) \, dx
\leq \frac{\displaystyle \vert Z \vert}{\displaystyle \vert X \vert}
\Vert u \Vert^2_{L^2(X)}
\end{equation}

\begin{equation} \label{ineq3} 
\displaystyle \int_{Z \setminus X_0} \bigg[ \frac{\displaystyle 1}
{\displaystyle 
\vert X \vert} \displaystyle \int_X u(y) \, dy \bigg]^2 \, dx
\leq \displaystyle \int_{Z \setminus X_0} \frac{\displaystyle 1}{\displaystyle
\vert X \vert} \bigg( \displaystyle \int_X u^2 \, dy \bigg) \, dx
\leq \frac{\displaystyle \vert Z \vert}{\displaystyle \vert X \vert}
\Vert u \Vert^2_{L^2(X)}
\end{equation}
and the Minkowski inequality, Eq. (\ref{ineq1}) can be rewritten as

\begin{eqnarray} \label{ineq4} \nonumber
&\Vert \tilde{u} \Vert^2_{L^2(Z)} \leq 
2 \displaystyle \int_{Z \cap X_0} (1-\Psi)^2 \, \vert u^* \vert^2 \, dx
+2 \displaystyle \int_{Z \cap X_0} \Psi^2 \, m^2 \, dx+
\frac{\displaystyle \vert Z \vert}{\displaystyle \vert X \vert}
\Vert u \Vert^2_{L^2(X)} \\
& \leq C(Z, X, \Psi) \Vert u \Vert^2_{L^2(X)}
\end{eqnarray}
For the derivative we obtain:

\begin{eqnarray} \label{ineq5} \nonumber
&\Vert \nabla \tilde{u} \Vert^2_{L^2(Z)}=\displaystyle \int_Z
\vert \nabla [ (1-\Psi) (u^*-m)] \vert^2 \, dx \\  
&=\displaystyle \int_Z \vert (u^*-m) \nabla (1-\Psi)+(1-\Psi) \nabla u^*
\vert^2 \, dx
\end{eqnarray}
By using the Minkowski and Poincar{\'e} inequalities, Eq. (\ref{ineq5})
becomes:

\begin{eqnarray} \label{ineq6} \nonumber
&\Vert \nabla \tilde{u} \Vert^2_{L^2(Z)} \leq
2 \displaystyle \int_{Z \cap X_0} \vert (u^*-m) \nabla (1-\Psi) \vert^2 \, dx
+2 \displaystyle \int_{Z \cap X_0} \vert (1-\Psi) \nabla u^* \vert^2 \,
dx \\
&\leq c \displaystyle \int_{Z \cap X_0} \vert (u^*-m)\vert^2 \, dx+
c^* \displaystyle \int_{Z \cap X_0} \vert \nabla u^* \vert^2 \, dx
\leq C_1(\Psi) \Vert \nabla u \Vert^2_{L^2(X)}
\end{eqnarray}

2. The construction of $\tilde{P}_{\epsilon}$ is obvious by summation
over the individual cells.

\end{proof}

\begin{lemma} \label{l6.1}

Let us consider the sets defined in Section \ref{sec1}.

\par\noindent
1. There exists a linear continuous extension operator

\begin{equation} \label{6.5b}
{P}: H^1(Z) \rightarrow H^1(Y)
\end{equation}
such that

\begin{equation}
{P} v=v \quad \text{in} \; Z
\end{equation}
and

\begin{equation} \label{6.6bb}
\Vert {P} \, v \Vert_{L^2(Y)} \leq C \Vert 
v \Vert_{L^2(Z)}
\end{equation}

\begin{equation} \label{6.6b}
\Vert \nabla ({P}\, v) \Vert_{L^2(Y)} \leq C \Vert \nabla 
v \Vert_{L^2(Z)}
\end{equation}
where $C$ is a positive constant.

\par\noindent
2. There exists a family of linear continuous extension operators

\begin{equation} \label{6.7b}
{P}_{\epsilon}: H^1(D_{\epsilon}) \rightarrow H^1(D)
\end{equation}
such that

\begin{equation}
{P}_{\epsilon} v^{\epsilon}=v^{\epsilon} \quad \text{in} \; D_{\epsilon}
\end{equation}
and

\begin{equation} \label{6.8bb}
\Vert {P}_{\epsilon} v^{\epsilon} \Vert_{L^2(D)} \leq C 
\Vert v^{\epsilon} \Vert_{L^2(D_{\epsilon})}
\end{equation}

\begin{equation} \label{6.8b}
\Vert \nabla ({P}_{\epsilon} v^{\epsilon}) \Vert_{L^2(D)} \leq C 
\Vert \nabla v^{\epsilon} \Vert_{L^2(D_{\epsilon})}
\end{equation}
where the constant $C>0$ does not depend on $\epsilon$.

\end{lemma}

\begin{proof}

This Lemma can be proved by applying the same arguments considered
in the proof of Lemma \ref{l6.2} and in Ref. \cite{NEUSS2} (p. $25$).

\end{proof}

\begin{remark}

Analogous extension theorems hold also for $u_i^{\epsilon}(t,x,z)$, defined
in $[0,T] \times \Pi_{\epsilon}$, and $v_i^{\epsilon}(t,x,z)$, defined
in $[0,T] \times \Omega_{\epsilon}$ ($1 \leq i \leq M$), since $z$ and $t$
can be considered as parameters, because of the geometry of the domain.

\end{remark}

From now on, we identify $v_i^{\epsilon}$ with its extension
$P_{\epsilon} v_i^{\epsilon}$ according to Lemma \ref{l6.1}
and $u_i^{\epsilon}$ with its extension
${\tilde P}_{\epsilon} u_i^{\epsilon}$ according
to Lemma \ref{l6.2}.
In order to study the limiting behavior of the set of
Eqs. (\ref{2.1a}), we use the notion of two-scale
convergence.
Some definitions and results on two-scale convergence, introduced
in Refs. \cite{Ngu}, \cite{All}, are reported in
Appendix C.

\begin{proposition} \label{p6.3}

Let $v_i^{\epsilon}$ ($1 \leq i \leq M$) be the extension of the solutions to
the system (\ref{2.1a}).
Then, up to a subsequence:

\begin{equation} \label{6.9a} 
v_i^{\epsilon} (t,x,z) \rightharpoonup v_i (t,x,z) 
\; \; \; \text{ weakly \, in} \; \; \; L^2(0,T; H^1(\Omega))
\end{equation}

\begin{equation} \label{6.9b}
\partial_t v_i^{\epsilon} (t,x,z) \rightharpoonup \partial_t v_i (t,x,z)  
\; \; \; \text{ weakly \, in} \; \; \; L^2([0,T] \times \Omega)
\end{equation}

\begin{equation} \label{6.9c}
v_i^{\epsilon} (t,x,z) \rightarrow v_i (t,x,z) \; \; \; \text{ strongly \, in}
\; \; \; C^0([0,T]; L^2(\Omega))
\end{equation}

\end{proposition}

\begin{proof}

The convergence results (\ref{6.9a})-(\ref{6.9c}) follow immediately from the
a priori estimates given in Lemma \ref{l3.2}.

\end{proof}

\begin{remark}

Since $v_i^{\epsilon} (t,x,z)$ converges weakly
to $v_i (t,x,z)$ ($1 \le i \le M$) in
$L^2(0,T; H^1(\Omega)) \cap H^1 (0,T; L^2(\Omega))$,
Theorem \ref{tC.4} in Appendix C implies the two-scale
convergence to
the same $v_i (t,x,z)$, and there exists a function ${\tilde v}_i \in
L^2([0,T] \times \Omega; H^1_{\sharp}(Y)/\mathbb{R})$ such that up to a
subsequence $\nabla_x v_i^{\epsilon}(t,x,z)$ two-scale converges to
$\nabla_x v_i (t,x,z)+\nabla_y \tilde{v}_i(t,x,y,z)$ (where $y \in Y$
is the microscopic variable).

Moreover, the interpolation-trace inequality (\ref{A.7})
in Appendix A and Theorem \ref{tC.6} in Appendix C
allow one to infer the two-scale convergence of
$v_i^{\epsilon}$ on the boundary $\Gamma_{\epsilon}$.

\end{remark}

Next we prove the convergence of $u_i^{\epsilon}$ ($1 \leq i \leq M$).

\begin{proposition} \label{p6.4}

Let $u_i^{\epsilon}(t,x,z)$ ($1 \leq i \leq M$) be the extension of the
solutions to the system (\ref{2.1a}).
Then, up to a subsequence:

\begin{eqnarray} \label{6.10} \nonumber
&u_i^{\epsilon} (t,x,z) \rightarrow {u}_i (t,x,y,z) \\ \nonumber
&\partial_t u_i^{\epsilon} (t,x,z) \rightarrow 
\partial_t {u}_i (t,x,y,z) \\ 
\nonumber
&\partial_z u_i^{\epsilon} (t,x,z) \rightarrow 
\partial_z {u}_i (t,x,y,z) \\
&\epsilon \nabla_x u_i^{\epsilon} (t,x,z) \rightarrow
\nabla_y {u}_i(t,x,y,z)
\end{eqnarray}
in the two-scale sense with
${u}_i(t,x,y,z) \in L^2([0,T] \times \Omega; H^1_{\sharp}(Y))
\cap H^1(0,T; L^2(\Omega \times Y))$ (where $y \in Y$ is the microscopic
variable).

\end{proposition}

\begin{proof}

The convergence results (\ref{6.10}) follow immediately from
the a priori estimates given in Lemma \ref{l3.2} and
Theorem \ref{tC.5} in Appendix C.

\end{proof}

\noindent

\textbf{Proof of the main Theorem \ref{t6.5}}

\noindent


In the case $m=1$, let us rewrite Eq. (\ref{2.8a}) in the form:

\begin{equation} \label{6.18a} 
\begin{split}
&\displaystyle \int_0^T \int_{\Omega} 
\partial_t v_1^{\epsilon} \,
\chi(\frac{\displaystyle x}{\displaystyle \epsilon})\,
\phi_1 \, dt \, dx \, dz+
d_1 \, \displaystyle \int_0^T \int_{\Omega} 
\nabla_x v_1^{\epsilon} \, \chi(\frac{\displaystyle x}{\displaystyle \epsilon})
\, \nabla_x \phi_1 \, dt \, dx \, dz \\ 
&+d_1 \, \displaystyle \int_0^T \int_{\Omega}
\partial_z v_1^{\epsilon} \, 
\chi(\frac{\displaystyle x}{\displaystyle \epsilon}) \, \partial_z \phi_1
\, dt \, dx \, dz 
-\epsilon \, \displaystyle \int_0^T \int_{\Gamma_{\epsilon}}
c_1(x,z) \, (u_1^{\epsilon}-v_1^{\epsilon})_+ \, \phi_1 \,
dt \, d\sigma_{\epsilon} \\
&=-\displaystyle \int_0^T \int_{\Omega} 
\bigg[ \sum_{j=1}^M b_{1,j} \, v_1^{\epsilon} \, v_j^{\epsilon} \bigg]\, 
\chi(\frac{\displaystyle x}{\displaystyle \epsilon}) \,
\phi_1   \, dt \, dx \, dz
\end{split}
\end{equation}
where $\chi(\frac{\displaystyle x}{\displaystyle \epsilon})$ is the
characteristic function of $\Omega_{\epsilon}$.
Inserting in Eq. (\ref{6.18a}) the following test function:

$$\phi_1 := \phi_1^0 (t,x,z)+ \epsilon
{\tilde \phi}_1 \bigg(t, x, \frac{x}{\epsilon},z \bigg)$$
where $\phi_1^0 \in C^1([0,T] \times \overline {\Omega})$ and
$\tilde{\phi}_1 \in C^1 ([0,T] \times \overline {\Omega}; C_{\#}^{\infty}(Y))$,
we obtain:

\begin{equation} \label{6.18} 
\begin{split}
&\displaystyle \int_0^T \int_{\Omega} 
\partial_t v_1^{\epsilon}(t,x,z) \,
\chi(\frac{\displaystyle x}{\displaystyle \epsilon}) \,
[\phi_1^0(t,x,z)+\epsilon \, \tilde{\phi}_1 (t,x, \frac{x}{\epsilon},z) ]
\, dt \, dx \, dz \\
&+d_1 \, \displaystyle \int_0^T \int_{\Omega} 
\nabla_x v_1^{\epsilon}  \,
\chi(\frac{\displaystyle x}{\displaystyle \epsilon}) \,
\bigg[ \nabla_x \phi_1^0 +\epsilon \nabla_x \tilde{\phi}_1 (t,x, 
\frac{x}{\epsilon},z)+
\nabla_{y} \tilde{\phi}_1 (t,x,\frac{x}{\epsilon},z) \bigg] \, dt 
\, dx \, dz\\
&+d_1 \, \displaystyle \int_0^T \int_{\Omega}
\partial_z v_1^{\epsilon} (t,x,z) \, 
\chi(\frac{\displaystyle x}{\displaystyle \epsilon}) \,
[\partial_z \phi_1^0 (t,x,z) +\epsilon \partial_z \tilde{\phi}_1 (t,x,
\frac{x}{\epsilon},z)] \, dt \, dx \, dz \\
&-\epsilon \, \displaystyle \int_0^T \int_{\Gamma_{\epsilon}}
c_1(x,z) \, (u_1^{\epsilon}(t,x,z)-v_1^{\epsilon}(t,x,z))_+ \, 
[\phi_1^0+\epsilon \tilde{\phi}_1(t,x,\frac{x}{\epsilon},z)] \,
dt \, d\sigma_{\epsilon}\\
&=-\displaystyle \int_0^T \int_{\Omega} 
\bigg[ \sum_{j=1}^M b_{1,j} \, v_1^{\epsilon}(t,x,z) \, v_j^{\epsilon} (t,x,z) 
\bigg]\, \chi(\frac{\displaystyle x}{\displaystyle \epsilon}) \,
[\phi_1^0+\epsilon \tilde{\phi}_1(t,x,\frac{x}{\epsilon},z)] \, dt 
\, dx \, dz
\end{split}
\end{equation}
Passing to the two-scale limit we get

\begin{equation} \label{6.19} 
\begin{split}
&\displaystyle \int_0^T \, dt \int_{\Omega} \, dx \, dz \int_{Z}
\partial_t v_1(t,x,z) \, 
\phi_1^0(t,x,z) \, d{y} \\
&+d_1 \, \displaystyle \int_0^T \, dt \int_{\Omega} \, dx \, dz \int_{Z}
[ \nabla_x v_1+\nabla_{y} \tilde{v}_1(t,x,{y},z) ] \,
[\nabla_x \phi_1^0+\nabla_{y} \tilde{\phi}_1(t,x,{y},z) ]
\, d{y} \\
&+d_1 \, \displaystyle \int_0^T \, dt \int_{\Omega} \, dx \, dz \int_{Z}
\partial_z v_1(t,x,z) \, \partial_z \phi_1^0(t,x,z) \, d{y}\\
&-\displaystyle \int_0^T \, dt \int_{\Omega} \, dx \, dz \int_{\Gamma} 
c_1(x,z) \, 
(u_1 (t,x,{y},z)-v_1(t,x,z))_+ \, \phi_1^0(t,x,z) \, d\sigma({y}) \\
&=-\displaystyle \int_0^T \, dt \int_{\Omega} \, dx \, dz \int_{Z} 
\bigg[ \sum_{j=1}^M b_{1,j} \, v_1 (t,x,z)\, v_j (t,x,z)\bigg]\, 
\phi_1^0(t,x,z) \,  d{y}
\end{split}
\end{equation}
An integration by parts shows that Eq. (\ref{6.19}) is a variational
formulation associated with the following homogenized system:

\begin{equation} \label{6.20}
-div_{y} [d_1 (\nabla_x v_1(t,x,z)+\nabla_{y} 
{\tilde v}_1(t,x,{y},z))]=0, \; \; \;
\; \; \; t>0, \; \; (x,z) \in \Omega, \; \; {y} \in Z 
\end{equation}
\begin{equation} \label{6.21}
[\nabla_x v_1(t,x,z)+\nabla_{y} {\tilde v}_1(t,x,{y},z)] \cdot 
\nu=0, \; \; \; \; \; \; \;
\; \; \;  \; \; \; \;  
t>0, \; \; (x,z) \in \Omega, \; \; {y} \in \Gamma
\end{equation}
\begin{equation} \label{6.22}
\begin{split}
&\vert Z \vert \frac{\displaystyle \partial v_1}{\displaystyle \partial t}
(t,x,z)-
div_x \bigg[ d_1 \displaystyle \int_{Z} d{y} \, (\nabla_x v_1(t,x,z)+
\nabla_{y} \tilde {v}_1(t,x,{y},z)) \bigg] \\ 
&-d_1 \, \vert Z \vert \, \partial_z^2 v_1(t,x,z)
+\vert Z \vert \sum_{j=1}^{M} b_{1,j} \, v_1(t,x,z) \, v_j(t,x,z) \\
&=\displaystyle \int_{\Gamma} c_1(x,z) \, (u_1(t,x,{y},z)-
v_1(t,x,z))_+ \, d\sigma(y) \;\;\;
\text{ in} \, \, [0,T] \times \Omega
\end{split}
\end{equation}
\begin{equation} \label{6.23}
\bigg[ \displaystyle \int_{Z} (\nabla_x v_1(t,x,z)+\nabla_{y} 
{\tilde v}_1(t,x,{y},z))\, d{y}
\bigg] \cdot \nu=0 \; \; \;  \qquad \; \; \; \; \; \; \; \; \; \; \; \;
\; \; \; \; \;
\text{ on} \, \, [0,T] \times \Gamma_L
\end{equation}
\begin{equation} \label{6.23a}
\partial_z v_1(t,x,z)=0 \, \, \, \qquad \; \; \; \; \; \; \; \;  \; \; \; 
\; \; \; \; \; \; \; \; \; \; \; \; 
\; \; \; \; \; \; \; \; \; \; \; \; \; \; \; \; \; \; \; \; \; \; \; \;
\; \; \; \; \; \; \; \; \; \;
\text{ on} \, \, [0,T] \times \Gamma_B
\end{equation}
By continuity we have that
\begin{equation} \label{6.24}
v_1 (t=0,x,z)=0 \, \, \, \qquad \; \; \; \; \; \; \; \;  \; \; \; \; \; \; \; \;
\; \; \; \; \; \; \; \; \; \; \; \; \; \; \; \; \; \; \; \; \; \; \; \;
\; \; \; \; \; \; \; \; \; \; \; \; \; \; \; \; \; \; \; \; \; \; \;
\; \; \; \;
\text{ in} \, \, \Omega.
\end{equation}

Since we have assumed that the diffusion coefficient is constant and
we have proved that the limiting function $v_1(t,x,z)$ does not depend
on the microscopic variable $y$,
Eqs. (\ref{6.20}) and (\ref{6.21}) reduce to:

\begin{equation} \label{6.25}
\triangle_{y} {\tilde v}_1(t,x,{y},z)=0, \, \, \, \qquad 
\; \; \; \; \; \; \; \; \; \;
\; \; \; \; \; \; \; \; \; \; \; \; \; \; \; \; \; \; \;
t>0, \; \; (x,z) \in \Omega, \; \; {y} \in Z 
\end{equation}
\begin{equation} \label{6.26}
\nabla_{y} {\tilde v}_1(t,x,{y},z) \cdot \nu=
-\nabla_x v_1(t,x,z) \cdot \nu,
\, \, \, \qquad \; \; \; \; \; 
t>0, \; \; (x,z) \in \Omega, \; \; {y} \in \Gamma
\end{equation}
Then, ${\tilde v}_1(t,x,{y},z)$ satisfying
Eqs. (\ref{6.25}) and (\ref{6.26}) can be written as

\begin{equation} \label{6.27}
{\tilde v}_1(t,x,{y},z)=\sum_{i=1}^3 \, w_i ({y}) \, 
\frac{\displaystyle \partial v_1}{\displaystyle \partial x_i}(t,x,z)
\end{equation}
where $(w_i)_{1 \leq i \leq 3}$
is the family of solutions of the cell problem

\begin{eqnarray} \label{6.28}
\begin{cases}
-\text{div}_{y} [\nabla_{y} w_i+ \hat{e}_i]=0 \, \, \, 
\qquad \; \; \; \; \; 
\text{ in} \, \, Z \\
(\nabla_{y} w_i+\hat{e}_i) \cdot \nu=0 \, \, \, \qquad \; \; \; \; 
\; \; \; \;
\text{ on} \, \, \Gamma \\
y \rightarrow w_i({y}) \; \; \; \; \; Y-\text{periodic}
\end{cases}
\end{eqnarray}
with $\hat{e}_i$ being the $i$th unit vector in ${\mathbb R}^3$.

Inserting the relation (\ref{6.27}) in Eqs. (\ref{6.22}) and (\ref{6.23}),
we get

\begin{equation} \label{6.29} 
\begin{split}
&\vert Z \vert \, \frac{\displaystyle \partial v_1}
{\displaystyle \partial t}(t,x,z)
-\text{div}_x \bigg[ d_1 \, A \, 
\nabla_x v_1(t,x,z) \bigg] \\
&-d_1 \, \vert Z \vert \, \partial_z^2 v_1(t,x,z) 
+\vert Z \vert 
\sum_{j=1}^M b_{1,j} \, v_1(t,x,z) \, v_j(t,x,z)  \\
&= \displaystyle \int_{\Gamma} c_1(x,z) \, (u_1(t,x,{y},z)-
v_1(t,x,z))_+ \, 
d\sigma(y) \, \, \, 
\; \; \; \text{ in} \, \, [0,T] \times \Omega
\end{split}
\end{equation}

\begin{equation} \label{6.30}
[ A \, \nabla_x v_1 (t,x,z)] \cdot \nu=0 \; \; \; \; \; \;  
\; \; \; \; \; \; \; \;
\text{ on} \, \, 
[0,T] \times \Gamma_{L}
\end{equation}
where $A$ is a matrix with constant coefficients defined by

$$A_{ij}=\displaystyle \int_{Z} (\nabla_{y} w_i+ \hat{e}_i) \cdot
(\nabla_{y} w_j+ \hat{e}_j) \, d{y}.$$

Let us now rewrite Eq. (\ref{2.7a}) as follows:

\begin{equation} \label{6.31} 
\begin{split}
&\displaystyle \int_0^T \int_{\Omega} 
\partial_t u_1^{\epsilon} \,
{\tilde \chi}(\frac{\displaystyle x}{\displaystyle \epsilon})\,
\psi_1 \, dt \, dx \, dz+
\epsilon^2 \, D_1 \, \displaystyle \int_0^T \int_{\Omega} 
\nabla_x u_1^{\epsilon} \,
{\tilde \chi}(\frac{\displaystyle x}{\displaystyle \epsilon}) \,
\nabla_x \psi_1  \, dt \, dx \, dz \\
&+{\tilde D}_1 \, \displaystyle \int_0^T \int_{\Omega}
\partial_z u_1^{\epsilon} \,
{\tilde \chi}(\frac{\displaystyle x}{\displaystyle \epsilon}) \,
\partial_z \psi_1  \, dt \, dx \, dz
+\epsilon \, \displaystyle \int_0^T \int_{\Gamma_{\epsilon}}
c_1(x,z) \, (u_1^{\epsilon}-v_1^{\epsilon})_+ \, \psi_1 \,
dt \, d\sigma_{\epsilon} \\
&=-\displaystyle \int_0^T \int_{\Omega} 
\bigg[ \sum_{j=1}^M a_{1,j} \, u_1^{\epsilon} \, u_j^{\epsilon} \bigg]\, 
{\tilde \chi}(\frac{\displaystyle x}{\displaystyle \epsilon})\,
\psi_1 \,  \, dt \, dx \, dz
+\displaystyle \int_0^T \int_{\Omega} f^{\epsilon} \, 
{\tilde \chi}(\frac{\displaystyle x}{\displaystyle \epsilon})\,
\psi_1 \,  \, dt \, dx \, dz
\end{split}
\end{equation}
where ${\tilde \chi}(\frac{\displaystyle x}{\displaystyle \epsilon})$
is the characteristic function of $\Pi_{\epsilon}$.
If we choose as test function:

$$\psi_1:= \tilde{\psi}_1 \bigg(t, x, \frac{x}{\epsilon},z \bigg)$$
where $\tilde{\psi}_1 \in C^1 ([0,T] \times \overline {\Omega}; 
C_{\#}^{\infty}(Y))$,
Eq. (\ref{6.31}) reads:

\begin{equation} \label{6.32} 
\begin{split}
&\displaystyle \int_0^T \int_{\Omega} 
\partial_t u_1^{\epsilon} (t,x,z)\,
{\tilde \chi}(\frac{\displaystyle x}{\displaystyle \epsilon}) \,
\tilde{\psi}_1(t,x, \frac{x}{\epsilon},z) \, dt \, dx \, dz \\
&+\epsilon^2 \, D_1 \, \displaystyle \int_0^T \int_{\Omega} 
\nabla_x u_1^{\epsilon} \,
{\tilde \chi}(\frac{\displaystyle x}{\displaystyle \epsilon}) \,
\bigg[ \nabla_x \tilde{\psi}_1(t,x, \frac{x}{\epsilon},z) 
+\frac{1}{\epsilon} \nabla_{y} 
\tilde{\psi}_1(t,x, \frac{x}{\epsilon},z) \bigg] \, dt 
\, dx\, dz \\
&+{\tilde D}_1 \, \, \displaystyle \int_0^T \int_{\Omega}
\partial_z u_1^{\epsilon} (t,x,z) \, 
{\tilde \chi}(\frac{\displaystyle x}{\displaystyle \epsilon}) \,
\partial_z \tilde{\psi}_1(t,x, \frac{x}{\epsilon},z) \,
dt \, dx \, dz \\
&+\epsilon \, \displaystyle \int_0^T \int_{\Gamma_{\epsilon}}
c_1(x,z) \, (u_1^{\epsilon}(t,x,z)-v_1^{\epsilon}(t,x,z))_+ \, 
{\tilde \psi}_1(t,x, \frac{x}{\epsilon},z) \, \,
dt \, d\sigma_{\epsilon}(x,z) \\
&=-\displaystyle \int_0^T \int_{\Omega} 
\bigg[ \sum_{j=1}^M a_{1,j} \, u_1^{\epsilon}(t,x,z) \, u_j^{\epsilon}(t,x,z) 
\bigg]\, 
{\tilde \chi}(\frac{\displaystyle x}{\displaystyle \epsilon}) \,
{\tilde \psi}_1(t,x, \frac{x}{\epsilon},z) \,  dt \, dx \, dz \\
&+\displaystyle \int_0^T \int_{\Omega} f^{\epsilon}(t,x,z) \,
{\tilde \chi}(\frac{\displaystyle x}{\displaystyle \epsilon}) \,
{\tilde \psi}_1(t,x, \frac{x}{\epsilon},z) \,  dt \, dx \, dz
\end{split}
\end{equation}
Passing to the two-scale limit we obtain:

\begin{equation} \label{6.33} 
\begin{split}
&\displaystyle \int_0^T \, dt \int_{\Omega} \, dx \, dz \int_{X}
\partial_t u_1(t,x,{y},z) \, 
\tilde{\psi}_1(t,x,{y},z) \, d{y} \\ 
&+D_1 \, \displaystyle \int_0^T \, dt \int_{\Omega} \, dx \, dz \int_{X}
\nabla_{y} u_1(t,x,{y},z) \,
\nabla_{y} \tilde{\psi}_1(t,x,{y},z) \, d{y}\\ 
&+\tilde{D}_1 \, \displaystyle \int_0^T \, dt \int_{\Omega} \, dx \, dz
\int_{X} \partial_z u_1(t,x,{y},z) \, 
\partial_z \tilde{\psi}_1(t,x,{y},z) \, d{y} \\
&+\displaystyle \int_0^T \, dt \int_{\Omega} \, dx \, dz \int_{\Gamma} 
c_1(x,z) \, 
(u_1 (t,x,{y},z)-v_1(t,x,z))_+ \, 
\tilde{\psi}_1(t,x,{y},z)\, d\sigma({y}) \\
&=-\displaystyle \int_0^T \, dt \int_{\Omega} \, dx \, dz  \int_{X} 
\bigg[ \sum_{j=1}^M a_{1,j} \, u_1(t,x,{y},z) \, 
u_j(t,x,{y},z) \bigg]\, \tilde{\psi}_1(t,x,{y},z) \, 
d{y} \\
&+\displaystyle \int_0^T \, dt \int_{\Omega} \, dx \, dz  \int_{X}
f(t,x,{y},z) \, \tilde{\psi}_1(t,x,{y},z) \,
d{y}
\end{split}
\end{equation}
An integration by parts shows that Eq. (\ref{6.33}) is a variational
formulation associated with the following homogenized system:

\begin{equation} \label{6.34} 
\begin{split}
&\frac{\displaystyle \partial u_1}{\displaystyle \partial t}(t,x,{y},z)-
D_1 \, \triangle_{y} u_1 (t,x,{y},z)
-\tilde{D}_1 \, \partial_z^2 u_1 (t,x,{y},z) \\
&= -\sum_{j=1}^M a_{1,j} \, u_1(t,x,{y},z)\, u_j(t,x,{y},z)
+f(t,x,{y},z), 
 \; t>0, \; (x,z) \in \Omega, \;  {y} \in X
\end{split}
\end{equation}
\begin{eqnarray} \label{6.35} \nonumber
&D_1 \, \nabla_{y} u_1 (t,x,{y},z) \cdot \nu \\
&=-c_1(x,z) \,
(u_1 (t,x,{y},z)-v_1 (t,x,z))_+,  \;  t>0, \;
(x,z) \in \Omega, \; \; {y} \in \Gamma
\end{eqnarray}
\begin{equation} \label{6.36}
\partial_z u_1 (t,x,{y},z)=0,\, \, \, \qquad t>0, \; \;
(x,z) \in \overline{D} \times \{0,L \}, \; \; {y} \in X
\end{equation}
To conclude, by continuity, we have that

\begin{equation} \label{6.37}
u_1 (t=0,x,{y},z)=U_1(x,{y},z) \, \, \, \qquad \; \; \; \; \; 
(x,z) \in \Omega, \; \; {y} \in X
\end{equation}

The proof in the case $1< m \leq M$ is achieved by applying exactly the same
arguments considered when $m=1$.


\section*{Acknowledgements}
S. L. is supported by GNFM of INdAM, Italy.


\appendix 

\section*{Appendix A} \label{appA}
\renewcommand{\theequation}{A.\arabic{equation}}
\setcounter{equation}{0}

In the following, we generalize the interpolation-trace inequality
from Ladyzenskaja \cite{Lady}(Eq. (2.21), p. 69):

\begin{equation} \label{A.1}
\Vert u \Vert^2_{L^2(\Gamma)} \leq C \Vert \nabla u \Vert_{L^2(\Omega)} \,
\Vert u \Vert_{L^2(\Omega)}
\end{equation}
which is valid for any function $u(x) \in H^1(\Omega)$ with:

\begin{equation} \label{A.2}
\displaystyle \int_{\Omega} u(x) \, dx=0
\end{equation}
where $\Gamma$ is an $(n-1)$-dimensional boundary of an $n$-dimensional
domain $\Omega$.

Let us consider now functions $u(x) \in H^1(\Omega)$
which do not satisfy (\ref{A.2}).
In this case, one can define

\begin{equation} \label{A.3}
u_0:=\frac{1}{\vert \Omega \vert} \displaystyle \int_{\Omega} u(x) \, dx
\end{equation}
and apply (\ref{A.1}) to $(u-u_0)$.
Hence,

\begin{eqnarray} \label{A.4} \nonumber
&\displaystyle \int_{\Gamma} \vert u(x) \vert^2 \, d\sigma(x)
\leq 2 \bigg[ \Vert u-u_0 \Vert^2_{L^2(\Gamma)}+\Vert u_0 
\Vert^2_{L^2(\Gamma)} \bigg] \\ \nonumber
&\leq 2 \, C \, \Vert \nabla (u-u_0) \Vert_{L^2(\Omega)} \Vert u-u_0 
\Vert_{L^2(\Omega)}
+\frac{\displaystyle 2 \vert \Gamma \vert }
{\displaystyle \vert \Omega , \vert^2} \bigg \vert \displaystyle \int_{\Omega} 
u(x) \, dx \bigg \vert^2  \\
&\leq
 2 \, C \bigg[ \eta \Vert \nabla u \Vert^2_{L^2(\Omega)}+
\eta^{-1} \Vert u-u_0 \Vert^2_{L^2(\Omega)} \bigg]+
\frac{\displaystyle 2 \vert \Gamma \vert }
{\displaystyle \vert \Omega \vert^2} \bigg \vert \displaystyle \int_{\Omega}
u(x) \, dx \bigg \vert^2
\end{eqnarray}
where the Young inequality has been used with $\eta$ being a
small constant.
By exploiting the Minkowski and H\"older inequalities, respectively, to
estimate the terms:

\begin{equation} \label{A.5}
\Vert u-u_0 \Vert^2_{L^2(\Omega)} \leq 2 \bigg[ \Vert u 
\Vert^2_{L^2(\Omega)}+\Vert u_0 \Vert^2_{L^2(\Omega)} \bigg]
\end{equation}

\begin{equation} \label{A.6}
\bigg \vert \displaystyle \int_{\Omega}u(x) \, dx \bigg \vert^2 \leq
\vert \Omega \vert \,  \Vert u \Vert^2_{L^2(\Omega)} 
\end{equation}
Eq. (\ref{A.4}) reads

\begin{equation} \label{A.7}
\Vert u \Vert^2_{L^2(\Gamma)} \leq 2 \, C \, \eta \Vert \nabla u
\Vert^2_{L^2(\Omega)}+\bigg[ 8 \, C \, \eta^{-1}+\frac{2 \vert \Gamma \vert}
{\vert \Omega \vert} \bigg] \Vert u \Vert^2_{L^2(\Omega)}
\end{equation}

\section*{Appendix B} \label{appB}

\renewcommand{\theequation}{B.\arabic{equation}}
\setcounter{equation}{0}

\renewcommand{\thetheorem}{B.\arabic{theorem}}

\begin{theorem}(Aubin-Lions-Simon) \label{tB.1}

Let $B_0 \subset B_1 \subset B_2$ be three Banach spaces.
We assume that the embedding of $B_1$ in $B_2$ is continuous and that the
embedding of $B_0$ in $B_1$ is compact.
Let $p,r$ such that $1 \leq p, r \leq +\infty$.
For $T>0$, we define

$$E_{p,r}=\{v \in L^p(]0,T[,B_0), \partial_t v \in L^r(]0,T[,B_2)\}.$$

\par\noindent
(i) If $p < +\infty$, $E_{p,r}$ is compactly embedded in $L^p(]0,T[,B_1)$.

\par\noindent
(ii) If $p=+\infty$ and if $r>1$, $E_{p,r}$ is compactly embedded in
$C^0([0,T],B_1)$.

\end{theorem}

\begin{theorem}(Lower-semicontinuity of the norm) \label{tB.2}

Let $E$ and $F$ be Banach spaces and $F'$ be the dual space of $F$.

\par\noindent
(i) Let $\{x_n\}$ be a sequence weakly convergent to $x$ in $E$.
Then, the norm on $E$ is lower semi-continuous with respect to the weak
convergence, i.e.

\begin{equation} \label {b3}
\Vert x \Vert_E \leq \lim \inf_{n \rightarrow \infty} \Vert x_n \Vert_E
\end{equation}

\par\noindent
(ii) Let $\{x_n\}$ be a sequence weakly$^*$ convergent to $x$ in $F'$.
Then, the norm on $F'$ is lower semi-continuous with respect to the
weak$^*$ convergence, i.e.

\begin{equation} \label {b4}
\Vert x \Vert_{F'} \leq \lim \inf_{n \rightarrow \infty} \Vert x_n \Vert_{F'}
\end{equation}

\end{theorem}

\section*{Appendix C} \label{appC}

\renewcommand{\theequation}{C.\arabic{equation}}
\setcounter{equation}{0}

\renewcommand{\thedefinition}{C.\arabic{definition}}
\renewcommand{\thetheorem}{C.\arabic{theorem}}

\begin{definition} \label{dC.1}
A sequence of functions $v^{\epsilon}$ in $L^2 ([0,T] \times \Omega)$
 two-scale converges to  $v_0 \in L^2 ([0,T] \times \Omega \times Y)$
if
\begin{equation} \label{c1}
\lim_{\epsilon \rightarrow 0} \int_0^T \int_{\Omega} v^{\epsilon}(t,x) \,
\phi \bigg( t,x,\frac{x}{\epsilon} \bigg) \,dt \,dx=
\int_0^T \int_{\Omega} \int_{Y} v_0 (t,x,y) \, \phi(t,x,y) \,dt \, dx\, dy
\end{equation}
for all $\phi \in C^1 ([0,T] \times \overline {\Omega}; C_{\#}^{\infty}(Y))$.
\end{definition}

\begin{theorem} \label{tC.2} (Compactness theorem)
If $v^{\epsilon}$ is a bounded sequence in $L^2 ([0,T] \times \Omega)$, then
there exists a function $v_0 (t,x,y)$ in $L^2 ([0,T] \times \Omega \times Y)$
such that, up to a subsequence,
$v^{\epsilon}$ two-scale converges to $v_0$.
\end{theorem}

\begin{theorem} \label{tC.3}
Let $v^{\epsilon}$ be a sequence of functions in $L^2 ([0,T] \times \Omega)$
which two-scale converges to a limit
$v_0 \in L^2 ([0,T] \times \Omega \times Y)$. Suppose, furthermore, that

\begin{equation} \label{c2} 
\lim_{\epsilon \rightarrow 0} \int_0^T \int_{\Omega} \vert v^{\epsilon}(t,x)
\vert^2 \, dt \, dx=
\int_0^T \int_{\Omega} \int_{Y} \vert v_0 (t,x,y) \vert^2 \, dt \, dx \, dy
\end{equation}
Then, for any sequence $w^{\epsilon}$ in $L^2 ([0,T] \times \Omega)$
that two-scale converges to a limit
$w_0 \in L^2 ([0,T] \times \Omega \times Y)$, we have

\begin{equation} \label{c3}\begin{split}
\lim_{\epsilon \rightarrow 0} \int_0^T \int_{\Omega} v^{\epsilon}(t,x) \, &
w^{\epsilon}(t,x) \,
\phi \bigg( t,x,\frac{x}{\epsilon} \bigg) \,dt \,dx
\\&=
\int_0^T \int_{\Omega} \int_{Y} v_0 (t,x,y) \, w_0 (t,x,y) \, 
\phi(t,x,y) \,dt \, dx\, dy
\end{split}\end{equation}
for all $\phi \in C^1 ([0,T] \times \overline {\Omega}; C_{\#}^{\infty}(Y))$.
\end{theorem}

In the following, we identify $H^1(\Omega)=W^{1,2} (\Omega)$, where the
Sobolev space $W^{1,p} (\Omega)$ is defined by
$$W^{1,p} (\Omega)=\bigg \{ v \vert v \in L^p(\Omega), 
\frac{\partial v}{\partial x_i} \in L^p(\Omega), i=1, 2, 3 \bigg \}$$
and we denote by $H^1_{\#}(Y)$ the closure of $C^{\infty}_{\#}(Y)$
for the $H^1$-norm.

\begin{theorem} \label{tC.4}
Let $v^{\epsilon}$ be a bounded sequence in
$L^2 ([0,T]; H^1 (\Omega))$ that converges weakly to a limit
$v(t,x)$ in $L^2 ([0,T]; H^1 (\Omega))$.
Then, $v^{\epsilon}$ two-scale converges to
$v(t,x)$, and there exists a function $v_1 (t,x,y)$ in
$L^2 ([0,T] \times \Omega; H^1_{\#} (Y)/\mathbb {R})$ such that, up to a
subsequence, $\nabla v^{\epsilon}$ two-scale converges to
$\nabla_x v(t,x)+\nabla_y v_1 (t,x,y)$.
\end{theorem}

\begin{theorem} \label{tC.5}
Let $v^{\epsilon}$ and $\epsilon \nabla v^{\epsilon}$ be two bounded
sequences in $L^2 ([0,T] \times \Omega)$.
Then, there exists a function $v_1 (t,x,y)$ in
$L^2 ([0,T] \times \Omega; H^1_{\#} (Y)/\mathbb {R})$ such that, up to a
subsequence, $v^{\epsilon}$ and $\epsilon \nabla v^{\epsilon}$
two-scale converge to $v_1 (t,x,y)$ and $\nabla_y v_1 (t,x,y)$,
respectively.
\end{theorem}

\begin{theorem} \label{tC.6}
Let $v^{\epsilon}$ be a sequence in $L^2 ([0,T] \times \Gamma_{\epsilon})$
such that

\begin{equation} \label{c4}
\epsilon \, \int_0^T \int_{\Gamma_{\epsilon}} \vert v^{\epsilon}(t,x) 
\vert^2 \, dt \, d\sigma_{\epsilon}(x) \leq C
\end{equation}
where $C$ is a positive constant, independent of $\epsilon$.
There exist a subsequence (still denoted by $\epsilon$) and a two-scale limit
$v_0(t,x,y) \in L^2([0,T] \times \Omega; L^2(\Gamma))$ such that
$v^{\epsilon}(t,x)$ two-scale converges to $v_0(t,x,y)$ in the sense that

\begin{equation} \label{c5} 
 \lim_{\epsilon \rightarrow 0} \, \epsilon \, \displaystyle \int_0^T 
\int_{\Gamma_{\epsilon}} v^{\epsilon}(t,x) \, \phi \bigg(t,x,
\frac{x}{\epsilon} \bigg) \, dt \,
d\sigma_{\epsilon}(x)  = 
 \displaystyle \int_0^T \int_{\Omega} \int_{\Gamma} v_0(t,x,y) \,
\phi(t,x,y) \, dt \, dx \, d\sigma(y) 
\end{equation}
for any function $\phi \in C^1 ([0, T] \times \overline {\Omega};
C_{\#}^{\infty}(Y))$.
\end{theorem}

\end{document}